\documentclass[12pt]{amsart}
\usepackage[left=1.5in, right=1.5in, top=1.5in, bottom=1.5in]{geometry}                
\newtheorem{thm}{Theorem}[section]
\newtheorem{cor}[thm]{Corollary}

\theoremstyle{definition}

\newtheorem{rem}[thm]{Remark}
\usepackage{graphicx}
\usepackage{amssymb}
\usepackage{amsmath}
\usepackage{epstopdf}
\usepackage{hyperref}
\title[The regular polygon]{On the first eigenvalue of the Laplacian for polygons}  
\author{Emanuel Indrei}
\address{Department of Mathematics\\
Purdue University\\
West Lafayette, Indiana \\
USA.}

\begin{document}
\setcounter{page}{1}
\pagenumbering{arabic}
\maketitle

\begin{abstract}

In 1947, P\'olya proved that if $n=3,4$ the regular polygon $P_n$ minimizes the principal frequency of an n-gon with given area $\alpha>0$ and suggested that the same holds when $n \ge 5$. In $1951,$ P\'olya \& Szeg\"o discussed the possibility of counterexamples in the book ``Isoperimetric Inequalities In Mathematical Physics." This paper constructs explicit $(2n-4)$--dimensional polygonal manifolds $\mathcal{M}(n, \alpha)$ and proves the existence of a computable $N \ge 5$ such that for all $n \ge N$, the admissible $n$-gons are given via $\mathcal{M}(n, \alpha)$ and there exists an explicit set $ \mathcal{A}_{n}(\alpha) \subset \mathcal{M}(n,\alpha)$ such that $P_n$ has the smallest principal frequency among $n$-gons in $\mathcal{A}_{n}(\alpha)$. Inter-alia when $n \ge 3$, a formula is proved for the principal frequency of a convex $P \in \mathcal{M}(n,\alpha)$ in terms of an equilateral $n$-gon with the same area; and, the set of equilateral polygons is proved to be an $(n-3)$--dimensional submanifold of the $(2n-4)$--dimensional manifold $\mathcal{M}(n,\alpha)$ near $P_n$. If $n=3$, the formula completely addresses a 2006 conjecture of Antunes and Freitas and another problem mentioned in ``Isoperimetric Inequalities In Mathematical Physics." Moreover, a solution to the sharp polygonal Faber-Krahn stability problem for triangles is given and with an explicit constant. The techniques involve a partial symmetrization, tensor calculus, the spectral theory of circulant matrices, and $W^{2,p}/BMO$ estimates. Last, an application is given in the context of electron bubbles.

\end{abstract}

\vskip .3in

\section{Introduction}

The principal frequency $\Lambda$ of a domain $\Omega$ is the frequency of the gravest proper tone of a uniform and uniformly stretched elastic membrane in equilibrium and fixed along the boundary $\partial \Omega$. In 1877, Lord Rayleigh observed that of all membranes with a given area, the circle generates the minimum principal frequency. This discovery was due to numerical evidence and a computation of the principal frequency of almost circular membranes. Mathematically, this property was known as the Rayleigh conjecture. Faber (1923) and Krahn (1925) found essentially the same proof for Rayleigh's conjecture.  

If $\Omega=B_1$ the principal frequency is the first positive root of the Bessel function 
$$
\Lambda=2.4048\ldots \hskip .05 in.
$$
Nevertheless, the principal frequency of a triangle is generally inaccessible via known formulas. In 1947 \footnote{The article was received Dec. 20, 1947 and published in 1948.} P\'olya proved that of all triangular membranes $T$ with a given area, the equilateral triangle has the lowest principal frequency \cite{MR26817}. A formula exists for the principal frequency of an equilateral triangle and therefore a lower bound on $\Lambda(T)$ is accessible. In the same article, P\'olya proved: of all quadrilaterals with a given area, the square has the smallest principal frequency. An essential tool in his proof is Steiner symmetrization: the set is transformed into a set of the same area with at least a line of symmetry. However, if one considers polygons with $n\ge 5$ sides, the technique in general increases the number of sides. P\'olya mentions in the article (p. $277$) that 
``It is natural to suspect that the propositions just proved about the equilateral triangle and square can be extended to regular polygons with more than four sides."  

The classical $1951$ book ``Isoperimetric Inequalities In Mathematical Physics" written by P\'olya and Szeg\"o mentions the problem in a less confident way \cite[p. 159]{pPZp}:
``to prove (or disprove) the analogous theorems for regular polygons with more than four sides is a challenging task." Afterwards, the problem was known as the P\'olya-Szeg\"o conjecture. In the book, there are several related problems. One of them is to prove the analog for the logarithmic capacity and this was solved in $2004$ by Solynin and Zalgaller \cite{MR2052355}. In 2006, Antunes and Freitas \cite{MR2264470} write about the principal frequency problem that ``no progress whatsoever has been made on this problem over the last forty years."

In addition, Henrot includes the principal frequency problem as Open problem 2, subsection 3.3.3 A challenging open problem in the book ``Extremum Problems for Eigenvalues of Elliptic Operators" \cite{MR2251558} where he writes (p. $51$)
``a beautiful (and hard) challenge is to solve the Open problem 2."

Recently, Bogosel and Bucur \cite{oeq4} showed that the local minimality of the convex regular polygon can be reduced to a single certified numerical
computation. For $n = 5, 6, 7, 8$ they performed this computation and certified the numerical approximation
by finite elements up to machine errors. However, since the formulation of the problem no theorem asserting minimality was proved.

My main theorem in this paper explicitly constructs sets for all sufficiently large $n$ such that the principal frequency is minimized by the regular convex $n$-gon in the collection of $n$-gons with the same area having vertices in these sets mod rotations and translations.  

\pagebreak
 \begin{thm}[A local set for which $P_n$ is the minimizer] \label{a}
 There exists a computable $N \ge 5$ such that for all $n \ge N$, if $P_n$ is the convex regular $n$-gon having vertices $\{w_i\}_{i=1}^n$ there exist explicit $\mathcal{A}_{n, i} \subset B_{a_n}(w_i)$, $a_n>0$, such that the minimizer of the principal frequency among $n$-gons with fixed area having vertices in $\bigotimes_{i=1}^{n}  \mathcal{A}_{n, i}$ (mod rigid motions) is $P_n$. 
\end{thm}

\begin{cor}[A global set for which $P_n$ is the minimizer] \label{pa}
There exists $N \ge 5$ and a modulus $\omega(0^+)=0$ such that for all $n \ge N$, if $P_n$ is the convex regular $n$-gon with $|P_n|=\pi$, 
$$
\Omega_{n}=\Big\{P' \in \mathcal{M}(n,\pi): \omega(|P' \Delta (B_1+a)|) \ge \lambda(P_n)-\lambda(B_1) \text{ for an $a \in \mathbb{R}^2$}\Big\} , 
$$

$$
Q_n= \Big\{P' \in \mathcal{M}(n,\pi): \text{vertices}(P') \in \bigotimes_{i=1}^{n}  \mathcal{A}_{n, i} \text{(mod a rigid motion)}\Big\},
$$
the minimizer of the principal frequency among $n$-gons with fixed area in $\mathcal{A}_{n} =\Omega_n \cup Q_n$ is $P_n$. 
\end{cor}

The method of proof involves tensor calculus and a partial (Steiner) symmetrization. For large $n$, since the area is fixed, if one considers a perturbation and a set of initial vertices $\{v_1,v_2,v_3\}$ in a clockwise-consecutive arrangement, the generated triangle is symmetrized with respect to the line intersecting the mid-point of the line-segment between $v_1$ and $v_3$ in a perpendicular fashion. Thus $|v_2^*-v_1|=|v_3-v_2^*|$, where $v_2 \mapsto v_2^*$. In order to investigate how the eigenvalue is affected, a flow $t \mapsto P_t$ is generated via the symmetrization. Therefore, the calculus which encodes the theory of (singular) moving surfaces is utilized. The process is iterated which yields a series depending on the first and second derivatives of the eigenvalues.

\begin{rem}
Observe that assuming $P \in \mathcal{M}(n,\alpha)$ is also convex, $\{P^k\}$ generated in the proof always converges to an equilateral $P^\infty \in \mathcal{M}(n,\alpha)$ and for $n \ge 3$. One may generate many explicit examples via $M_n(T(n)) \subset \mathcal{A}_{n}$ in the main proof. Let $a \in \mathbb{N}$ \& choose $n\ge N$ sufficiently large such that $T(n) \ge a$; hence if the convergence is in $a$ iterations or fewer, then $\lambda_n(P) \ge \lambda_n(P^\infty)$. In particular, the minimization encodes the equilateral polygons, and if $P^\infty=P_n$, the minimality is true. Therefore letting for instance $a=2022^{2022^{2022}}$ implies lots of examples. The general minimization improvement to equilateral polygons in \eqref{ea} is with fewer assumptions; more precisely, a bound via the rate explicit in $Q_n$ and the localization, thus, combined with the global enhancement in Corollary \ref{pa} the minimization improvement generates a large collection of polygons. Interestingly, one may prove the rate for triangles \S \ref{pq3} and the localization is simple to generalize thanks to comparison arguments with rectangles, hence the space constructed when $n$ is large is natural and the optimal assumptions for the rate may have formulations in terms of angle restrictions. 

Nevertheless, the unique minimizer of the polygonal isoperimetric inequality in the class of convex $n$-gons is the regular, and in general, the convex regular. Since an example of a regular polygon in a more general class is the pentagram, the convexity is necessary. Thus when the polygon is convex a necessary and sufficient condition is cyclicity and equilaterality. Observe that the sequence $\{P^k\}$ converges to an equilateral polygon $P^\infty \in \mathcal{M}(n,\alpha)$ when $P^1 \in \mathcal{M}(n,\alpha)$: therefore the missing characteristic is cyclicity; but, if $n$ is large, $P_n$ is close to a disk, thus the cyclicity shows up naturally \& Corollary \ref{pa} localizes the problem for $n$ large to a neighborhood of $P_n$. Now since in Theorem \ref{a}, the formula I proved (see the proof) is true without the rate in $Q_n$, the remaining parts to completely prove that the global minimizer is $P_n$ are: (i) to estimate the radius of the neighborhood and reduce the complete problem to a neighborhood where one always has convexity via the non-degenerate convex $P_n$; one way of investigating this is to explicitly identify the modulus $\omega$ which is (modulo a non-explicit constant) quadratic \cite{MR3357184}; nevertheless, for some subsets, the modulus could be much better than quadratic; (ii) when localizing, to prove that when $n$ is large, many polygons in the neighborhood converge via the partial symmetrization to the regular (and therefore convex) $n$-gon and utilize the formula to then compare the eigenvalues; the formula is encoded more generally in Theorem \ref{E}. The second derivative via the localization is for non-trivial iterations always positive, therefore one has to investigate the first derivative carefully and exploit the symmetry to obtain a fast decay (I used the rate, symmetry, $W^{2,p}/BMO$ theory, and the polygonal isoperimetric stability to show this decay).
The manifold $\mathcal{M}(n,\alpha)$ has non-convex polygons and the equilateral polygons near $P_n$ generate an $(n-3)-$dimensional submanifold although there are non-convex equilateral $n$-gons which when taking the radius of the neighborhood around $P_n$ small, vanish via the non-degenerate convexity of $P_n$. This then hints towards a complete solution for $n$ large in case the general minimization is as suggested by P\'olya in his paper \cite{MR26817}. Last, to remove the largeness assumption on $n$, the constants appearing in the proof of Theorem \ref{a} have a fundamental role via explicit estimates. In spite of the preclusion of cyclicity for small values of $n$, when $n=4$ for example, the limit is a rhombus which after one Steiner symmetrization is changed into a rectangle and then the eigenvalue is explicit and may be compared to the eigenvalue of a square. Thus for low values of $n$, one has to understand a new way of evolving $P^\infty$ into a more cyclical polygon, see \S \ref{ra7}, \S \ref{z1k}.
\end{rem}

\pagebreak

\begin{thm} \label{E}
Assume $n \ge 3$ and $P \in \mathcal{M}(n,\alpha)$ is convex. Then there exists an equilateral $P_{\text{eq}} \in  \mathcal{M}(n,\alpha)$ such that

$$
 \lambda(P)=\lambda(P_{\text{eq}})+\sum_{k=2}^\infty \alpha_kt_{k-1}+\sum_{k=2}^\infty  \beta_k\frac{t_{k-1}^2}{2},
$$
where $P^1=P$, $P^{k}=(P^{k-1})^{*}$ represents the n-gon constructed from $P^{k-1}$ as in the proof of Theorem \ref{a}, 

$$
\alpha_k=\frac{d \lambda(P^{k})}{dt}\Big|_{t=0}=\int_0^{\xi_k} \frac{\alpha}{\xi_k}|\nabla u_{n,k}(\alpha, y_{-})|^2 d\alpha-\int_0^{\xi_k} \frac{\alpha}{\xi_k} |\nabla u_{n,k}(\alpha, y_+)|^2 d\alpha,
$$

\begin{align*}
\beta_k=\frac{d^2 \lambda(P^{k})}{d t^2}\Big|_{t=t_{e_{k-1}}}&=\frac{2(\frac{b_k}{2}-t_{e_{k-1}})}{\xi_k(\xi_k^2+(t_{e_{k-1}}-\frac{b_k}{2})^2)}\int_0^{\xi_k} \alpha |\nabla u_{n,t_{e_{k-1}}}(\alpha, y_+)|^2 d \alpha\\
&+\frac{2(\frac{b_k}{2}+t_{e_{k-1}})}{\xi_k(\xi_k^2+(t_{e_{k-1}}+\frac{b_k}{2})^2)}\int_0^{\xi_k} \alpha |\nabla u_{n,t_{e_{k-1}}}(\alpha, y_{-})|^2 d \alpha,
\end{align*}
$y_{\pm}, \xi_k, b_k, t_{k-1}$ are calculated explicitly from $P^{k-1}$, $t_{e_{k-1}} \in [0, t_{k-1}]$, and $u_{n,k}$ $\&$ $u_{n,t_{e_{k-1}}}$ denote the corresponding eigenfunctions. Furthermore, the set of equilateral polygons with area $\alpha$, $\mathcal{E}(n, \alpha)$, is an $(n-3)$--dimensional submanifold of the $(2n-4)$--dimensional manifold $\mathcal{M}(n,\alpha)$ near $P_n$. 
\end{thm}

In $1951,$ P\'olya \& Szeg\"o \cite[vii]{pPZp} discussed the problem of finding an explicit formula for the eigenvalue of a triangle. The theorem addresses this: e.g. via Corollary \ref{Y}, Remark \ref{E_p}, Corollary \ref{E_3}, Corollary \ref{E_4}, \& Remark \ref{E_5}.
\begin{cor} \label{Y}
The principal frequency of a triangle $T$ with given area $A$ is

$$
\Lambda(T)=\sqrt{\frac{4 \pi^2}{A\sqrt{3}}+ \sum_{k=2}^\infty  \alpha_k \frac{t_{k-1}^2}{2}},
$$
where $T^1=T$, $T^{k}=(T^{k-1})^{*}$ represents the triangle constructed from $T^{k-1}$ as in the proof of Theorem \ref{a} (when $n=3$),

\begin{align*}
\alpha_k=\frac{d^2 \lambda(T^{k})}{d t^2}\Big|_{t=t_{e_{k-1}}}&=\frac{2(\frac{b_k}{2}-t_{e_{k-1}} )}{\xi_k(\xi_k^2+(t_{e_{k-1}} -\frac{b_k}{2})^2)}\int_0^{\xi_k} \alpha |\nabla u_{3, t_{e_{k-1}} }(\alpha, y_+)|^2 d \alpha\\
&+\frac{2(\frac{b_k}{2}+t_{e_{k-1}} )}{\xi_k(\xi_k^2+(t_{e_{k-1}} +\frac{b_k}{2})^2)}\int_0^{\xi_k} \alpha |\nabla u_{3, t_{e_{k-1}} }(\alpha, y_{-})|^2 d \alpha,
\end{align*}
$y_{\pm}, \xi_k, b_k, t_{k-1}$ are calculated explicitly from $T^k$,  $t_{e_{k-1}} \in [0, t_{k-1}]$, and $u_{3, t_{e_{k-1}}}$ is the corresponding eigenfunction.

\end{cor}
\begin{rem} \label{E_p}
The eigenfunctions around the vertex of a polygon with opening $\pi / \alpha$ have the form
$$
\alpha_ar^{\alpha}\sin(\alpha \phi)+O(r^{2\alpha})+O(r^{2+\alpha}),
$$ 
$\phi \in [0,\frac{\pi}{\alpha}]$ in the polar coordinates relative to the cone. Note that $\frac{d^2 \lambda}{d t^2}$ also contains $u_{3, t_{e_{i-1}}}$ but it is evaluated on the edges of triangles obtained via the construction mentioned in the proof of Theorem \ref{a}. Therefore, one has additional information via the angles of the triangles which generate more explicit formulas for $\frac{d^2 \lambda}{d t^2}$.
\end{rem}

\begin{cor} \label{E_3}
If $T$ is a  triangle with given area $A$, there exists an isosceles triangle $T_{iso}$ such that

$$
\Lambda(T)=\sqrt{\lambda(T_{iso})+ \alpha_1 t^2},
$$
where $T_{iso}=T^{*}$ represents the triangle constructed from $T$ as in the proof of Theorem \ref{a} (when $n=3$), 
\begin{align*}
\alpha_1&=\frac{2(\frac{b_k}{2}-t_{e_{k-1}})}{\xi_k(\xi_k^2+(t_{e_{k-1}}-\frac{b_k}{2})^2)}\int_0^{\xi_k} \alpha |\nabla u_{3,t_{e_{k-1}}}(\alpha, y_+)|^2 d \alpha\\
&+\frac{2(\frac{b_k}{2}+t_{e_{k-1}})}{\xi_k(\xi_k^2+(t_{e_{k-1}}+\frac{b_k}{2})^2)}\int_0^{\xi_k} \alpha |\nabla u_{3,t_{e_{k-1}}}(\alpha, y_{-})|^2 d \alpha,
\end{align*}
$y_{\pm}, \xi_k, b_k, t$ are calculated explicitly from $T$, $t_{e_{k-1}} \in [0, t]$, and $u_{3,t_{e_{k-1}}}$ is the corresponding eigenfunction.
\end{cor}

\begin{cor} \label{E_4}
If $T=ABC$ is a  triangle with given area $\rho$ and sides $c \le a \le b$, there exists an isosceles triangle $T_{iso}$ with area $\rho$ such that

$$
\Lambda(T)=\sqrt{\lambda(T_{iso})+ \alpha_1 t^2+o(t^2)},
$$
where $T_{iso}=T^{*}$ represents the triangle constructed from $T$ as in the proof of Theorem \ref{a} (when $n=3$), $t=\frac{b}{2}-\frac{\overrightarrow{AB} \cdot \overrightarrow{AC}}{b}$,
$$
\alpha_1= \frac{b^2}{\rho ((\frac{b}{2})^2+(\frac{2\rho}{b})^2)} \int_0^{\frac{2\rho}{b}} \alpha |\nabla u(\alpha, y_+)|^2 d \alpha,
$$
$y_{+}$ is calculated explicitly from $T$, and $u$ is the eigenfunction of $T_{iso}$.

\end{cor}

\begin{rem} \label{E_5}
One may identify $\lambda(T_{iso})$ in large classes of triangles via formulas in e.g. \cite{MR2264470op}: assume $T_{iso}$ has sides $a < b$ with $\alpha$ the angle corresponding to $a$; the eigenvalue then can be computed via

$$
\lambda(T_{iso})=\frac{\pi^2}{\alpha^2 b^2}-\frac{2^{2/3} a_1 \pi^{4/3}}{\alpha^{4/3} b^2} + 4\frac{2^{1/3} a_1^2 \pi^{2/3}}{5 \alpha^{2/3} b^2}+\Big(\frac{37}{70}-\frac{52}{175}a_1^3+\frac{1}{6} \pi^2 \Big) \frac{1}{b^2} + O(\alpha^{2/3}),
$$
where $a_1$ is associated with a zero of the Airy function of the first kind and $\alpha$ is small. Moreover, as mentioned in Remark \ref{E_p} the constant $\alpha_1$ can be calculated in a more explicit way via the angles of the triangles and the eigenfunction representation around the vertices. Note that a quick bound is:

$$
\alpha_1 \le ||\nabla u(\alpha, y_+(\alpha))||_{L^\infty([0, \frac{2\rho}{b}])}^2 \frac{2\rho}{(\frac{b}{2})^2+(\frac{2\rho}{b})^2}.
$$

\end{rem}

Inter-alia to constructing the space of polygons in Theorem \ref{a}, the equivalence of the scaling-invariant principal frequency deficit and the scaling-invariant polygonal isoperimetric deficit naturally appears. The equivalence completely addresses a 2006 conjecture of Antunes and Freitas \cite[p. 338]{MR2264470}: Corollary \ref{Yzk6} and Remark \ref{Yzk8}.

\begin{thm} \label{Yzk}
There exist computable $a_k>0$ where $k \in \{0,1\}$ such that if $P$ is a triangle, $P_3$ is an equilateral triangle,  

$$
\delta_{\lambda}(P):= |P|\lambda(P)-|P_3|\lambda(P_3),
$$ 

$$
\delta_{\mathcal{P}}(P):=\frac{L^2(P)}{12 \sqrt{3}|P|}-1,
$$

then

$$
a_0 |P_3|^{3/2} \delta_{\mathcal{P}}(P) \le \delta_{\lambda}(P) \le a_1|P_3|^{3/2} \delta_{\mathcal{P}}(P) .
$$

\end{thm}

\begin{cor} \label{Yzk6}
Let $P \in \mathcal{M}(3,A)$ and suppose $a_0$ $\&$ $a_1$ are the constants in Theorem \ref{Yzk}. 
Set
$$
\overline{\theta}_{\alpha}:=\frac{a_0}{12 \sqrt{3}}
$$

$$
\overline{\theta}_{\mathcal{A}}:=\frac{a_1}{12 \sqrt{3}},
$$
then
$$
\frac{4\pi^2}{A\sqrt{3}}+ \overline{\theta}_{\alpha} \frac{L^2-12 \sqrt{3}A}{A^2} \le \lambda(P) \le \frac{4\pi^2}{A\sqrt{3}}+\overline{\theta}_{\mathcal{A}}\frac{L^2-12 \sqrt{3}A}{A^2}.
$$
Moreover, if $0 < \epsilon<1$, there exists $P_\epsilon \in \mathcal{M}(3,1)$  such that 
 $$
  \frac{\pi^2}{16}(1-\epsilon)< \frac{\lambda(P_\epsilon)-\frac{4\pi^2}{A\sqrt{3}}}{L^2(P_\epsilon)-12\sqrt{3}}< \frac{\pi^2}{16}\frac{1}{(1-\epsilon)^2}.
 $$
\end{cor}

\begin{rem} \label{Yzk8}
Note that when $\epsilon$  is selected very close to $0$, 

$$
\frac{\lambda(P_\epsilon)-\frac{4\pi^2}{A\sqrt{3}}}{L^2(P_\epsilon)-12\sqrt{3}}
$$ 
is very close to $\frac{\pi^2}{16}.$ 
In particular 

$$
\overline{\theta}_{\alpha} \le  \frac{\pi^2}{16}.
$$
The value $\frac{\pi^2}{16}$ is less than a conjectured value $\approx .67$ \cite[p. 339]{MR2264470}. Observe in addition that the constants $\overline{\theta}_{\alpha}, \overline{\theta}_{\mathcal{A}}$ can be estimated via the constants that appear in the proof of Theorem \ref{Yzk}. Moreover, the conjecture contains the two inequalities in \cite[(5-2), p. 338]{MR2264470} and \cite[Conjecture 5.1., p. 339]{MR2264470} is already proven \cite{MR2369934} which includes the first of those which is the upper bound with the sharp constant $\frac{\pi^2}{9}$, but the method utilizes Mathematica. The result in Corollary \ref{Yzk6} contains both upper and lower bounds and is independent of \cite{MR2369934} (furthermore, the proof does not depend on a computer); some other inequalities were proved in \cite{MR2779073}.

\end{rem}

Also, the equivalence (more specifically, the lower bound in Corollary \ref{Yzk6}) solves the sharp polygonal Faber-Krahn stability problem for triangles with an explicit constant: Corollaries \ref{Yz} \& \ref{Yzo}. The sharp Faber-Krahn stability problem with a non-explicit constant was solved in \cite{MR3357184}; the sharp isoperimetric stability problem with a non-explicit constant was solved in \cite{MR2456887} and with an explicit constant in \cite{MR2672283}.
One interesting fact is that the simpler-to-state polygonal isoperimetric inequality stability problem with an explicit constant was solved by the author and Nurbekyan \cite{MR3327086} after the more general isoperimetric stability.

\begin{cor} \label{Yz}
There exists a computable $a>0$ such that if $T$  is a triangle 
$$
|T|\lambda(T)-\lambda(T_{Eq}) \ge a \Big(\frac{|R(T) \Delta sT_{Eq}|}{|T|}\Big)^2,
$$
where $T_{Eq}$ is the equilateral triangle with area $1$, $R$ a rigid motion, $s=\sqrt{|T|}$, and the exponent $2$ is sharp. 
\end{cor}

With 

$$
\mathcal{A}(T):=\inf \Big\{\frac{|R(T) \Delta sT_{Eq}|}{|T|}: \text{ where $R$ is a rigid motion $\&$ $|sT_{eq}|=|T|$}\Big\}
$$
(the Fraenkel asymmetry), the previous then yields:

\begin{cor} \label{Yzo}
There exists an explicit $a>0$ such that if $T$  is a triangle and $T_{eq}$ is an equilateral triangle 
$$
|T|\lambda(T)-|T_{eq}|\lambda(T_{Eq}) \ge a \mathcal{A}^2(T),
$$
and the exponent $2$ is sharp. 
\end{cor}

\begin{rem}
One can construct a simple example of a pair of non-isometric simply-connected
domains in the euclidean plane which are isospectral \cite{MR1181812}. Another construction, discovered by Milnor, exhibits
two $16-$dimensional toruses which are distinct as Riemannian manifolds but have
the same sequence of eigenvalues \cite{MR162204}. Nevertheless, observe that one can hear the shape of an equilateral triangle and the stability thus yields that one can compare frequencies and detect a near-equilateral triangle. Also, one only requires one eigenvalue, not all of them!
\end{rem}

Observe that Theorem \ref{E} also generates a formula for the eigenvalue of a quadrilateral. The limit $P^\infty$ is a rhombus and therefore there is some additional information which allows a more detailed expression: more specifically, the rhombus can be changed into a rectangle with one Steiner symmetrization and the tensor theory can be utilized for the eigenvalue calculation.

 \begin{thm} \label{Eu}
Assume $P \in \mathcal{M}(4,\alpha)$ is convex. Then there exists $\Psi: \mathbb{R}^+ \rightarrow \mathbb{R}^+$ 
with
$$
\lim_{t \rightarrow 0^+} \frac{\Psi(t)}{t^2}=0 
$$ 
such that 

$$
 \lambda(P)=\pi^2\Big(\frac{1}{L^2}+\frac{1}{l^2}\Big)+\sum_{k=2}^\infty \alpha_kt_{k-1}+\sum_{k=2}^\infty  \beta_k\frac{t_{k-1}^2}{2}+ \Psi(a),
$$
where $P^1=P$, $P^{k}=(P^{k-1})^{*}$ represents the n-gon constructed from $P_{k-1}$ as in the proof of Theorem \ref{a}, 
$$
\alpha_k=\frac{d \lambda(P^{k})}{dt}\Big|_{t=0}=\int_0^{\xi_k} \frac{\alpha}{\xi_k}|\nabla u_{4,k}(\alpha, y_{-})|^2 d\alpha-\int_0^{\xi_k} \frac{\alpha}{\xi_k} |\nabla u_{4,k}(\alpha, y_+)|^2 d\alpha,
$$

\begin{align*}
\beta_k=\frac{d^2 \lambda(P^{k})}{d t^2}\Big|_{t=t_{e_{k-1}}}&=\frac{2(\frac{b_k}{2}-t_{e_{k-1}})}{\xi_k(\xi_k^2+(t_{e_{k-1}}-\frac{b_k}{2})^2)}\int_0^{\xi_k} \alpha |\nabla u_{4, t_{e_{k-1}}}(\alpha, y_+)|^2 d \alpha+\\
&\frac{2(\frac{b_k}{2}+t_{e_{k-1}})}{\xi_k(\xi_k^2+(t_{e_{k-1}}+\frac{b_k}{2})^2)}\int_0^{\xi_k} \alpha |\nabla u_{4, t_{e_{k-1}}}(\alpha, y_{-})|^2 d \alpha,
\end{align*}
$y_{\pm}, \xi_k, b_k, t_{k-1}, l, L, P_R$ are calculated explicitly from $\{P^k\}$ with $P_R$ a rectangle with area $\alpha$ $\&$ sides $l, L$ constructed via one Steiner symmetrization applied to the rhombus $P^\infty$,  $t_{e_{k-1}} \in [0, t_{k-1}]$, $u_{4, k}$ $\&$ $u_{4, t_{e_{k-1}}}$ denote the corresponding eigenfunctions, and $0 \le a \le a_*dist(P_R, P^\infty)$ for a universal $a_*>0$.
\end{thm}

Interestingly, an application of Theorem \ref{a} \& Corollary \ref{pa} naturally appears in the context of electron bubbles \cite{MR2264470nop}. Electron bubbles form when electrons enter into liquid helium: the electrons repel the helium atoms and form areas (cavities) free of helium. The equilibrium is obtained in terms of minimizing
$$
E=\Psi \lambda+\Sigma\int_S dS+\Pi \int_\Omega d\Omega,
$$
with $\Sigma$ the surface tension density, $\Pi$ the hydrostatic pressure, $\lambda$ the eigenvalue, $\Psi = \frac{h^2}{8\pi^2 m}$, $h$ Planck's constant, and $m$ the electron's mass. 

The technique of generating $\mathcal{A}_{n}(\alpha)$ illuminates the bubble equilibrium which may directly be computed by a scaling argument thanks to Corollary \ref{pa} and the polygonal isoperimetric inequality. 

\begin{cor}[of Theorem \ref{a}] \label{o}
Assume $P_n$ is the convex regular polygon with area $\pi$. There exists $N \ge 5$ such that for all $n \ge N$,   
$aP_n$ minimizes the energy

\begin{align*}
E&=\Psi \lambda+\Sigma \int_S dS+\Pi \int_\Omega d\Omega
\end{align*}
in 

$$
\bigcup_{\alpha>0} \mathcal{A}_{n}(\alpha),
$$
and $a$ satisfies

$$
\frac{-2 \Psi \lambda(P_n)}{\Sigma \mathcal{H}^1(\partial P_n)}+a^3+\frac{2\Pi \pi}{\Sigma \mathcal{H}^1(\partial P_n)}a^4=0.
$$
\end{cor}

\subsection{Applications}

The initial application involves the theory of sound. This was the subject in Lord Rayleigh's books \cite{pPZpoq}. P\'olya and Szeg\"o were studying elasticity problems and wrote \cite[vii]{pPZp}: ``The results hitherto obtained and discussed in this book allow already in some cases a fairly close estimate of physical quantities in which the engineer or the physicist may have a practical interest. And it seems possible to follow much further the road here opened." 

\subsubsection{Droplets}
The equilibrium configuration of an isolated droplet subject to only the physical phenomenon of surface tension is a sphere and is stable with respect to small perturbations of the droplet's surface. In the case when the total potential energy is minimized and the dynamics of the system are ignored, the Navier-Stokes/Euler equations are not employed. The droplet also may not need to be a fluid; it could be an isotropic solid or crystal \cite{pPZ8p}. The energy is 
$$
E=\int_S\sigma dS-\lambda\int_\Omega d\Omega.
$$  
The variation is
$$
\delta E = -\int_S C(\sigma B_\alpha^\alpha+\lambda)dS
$$
and thus the equilibrium condition is $B_\alpha^\alpha = \text{Constant}$, where $B_\alpha^\alpha$ is the mean curvature; therefore, the solution is a surface of constant curvature. The sphere, a set of two disjoint spheres, and also a shape assumed by a thin film supported by a wire loop satisfy the condition. Hence the constant curvature equation is complex. The idea to generate stable configurations is to have gradient descent with volume preservation: choosing a specific evolution by specifying $C$; e.g. $C=-\nabla_j\nabla^j B_\alpha^\alpha$. The example leads to 
$$
\frac{d E_{ST}}{dt}=-\int_S \nabla_jB_\beta^\beta\nabla^j B_\alpha^\alpha dS \le 0.
$$    
The stability properties of the equilibrium are then analyzed with the second variation.

\subsubsection{Electron Bubbles}
The droplet example may be modified via an eigenvalue in the energy \cite{MR2264470nop}. The equilibrium is obtained as already stated in terms of minimizing
$$
E=\frac{h^2}{8\pi^2 m} \lambda+\Sigma\int_S dS+\Pi\int_\Omega d\Omega,
$$
with $\Sigma$ the surface tension density, $\Pi$ the hydrostatic pressure, $\lambda$ the eigenvalue, $\Psi = \frac{h^2}{8\pi^2 m}$, $h$ Planck's constant, and $m$ the electron's mass. The excited states of the electrons correspond to different eigenvalues of the (negative) Laplacian and a problem in applied sciences is to study the equilibrium shapes for each eigenvalue.  The first eigenvalue corresponds to the ground state and this is stable. A surprising result is that the radially symmetric energy state corresponding to the second eigenvalue is morphologically unstable. The first variation is
$$
\delta E = \int_S C(-\Psi \nabla^j \Psi \nabla_j \Psi -\Sigma B_\alpha^\alpha+\Pi) dS;
$$ 
$C$ is an independent variation and this leads to
$$
\Psi  \nabla^j \Psi \nabla_j \Psi +\Sigma B_\alpha^\alpha-\Pi=0.
$$
If the surface is a sphere, the equation leads to an algebraic equation for the radius of the equilibrium. The stability is analyzed again via the second variation and in general is complicated, thus only computed at equilibrium configurations
$$
\delta^2 E = \int_S C\Big(2 \Psi \nabla^j \frac{\delta \Psi}{\delta t} \nabla_j \Psi+\Sigma(\nabla^\alpha \nabla_\alpha C+C B_\beta^\alpha B_\alpha^\beta) \Big)dS.
$$
Simulations of electron bubbles in terms of pressure are given in \cite{MR2264470nop}.

\subsection{Tools}
One way of analytically understanding the principal frequency is to consider a function $u$ defined in the interior of $\Omega$ and vanishing on $\partial \Omega$. The defining property of $\Lambda^2$ is:

$$
\frac{\int_\Omega |\nabla u|^2 d\Omega}{\int_\Omega u^2 d\Omega} \ge \Lambda^2.
$$
Equality holds if and only if $u=a w$, 

$$
\Delta w+\Lambda^2 w = 0;
$$
without loss, $w>0$ in $\Omega$. This $w$ characterizes the shape of a membrane when it vibrates emitting its deepest tone. 

If $S_t$ is an evolving surface and $A$ an object defined on $S_t$, the invariant derivative is \cite{MR2264470nop, Caltens}:

$$
\frac{\delta A(w)}{\delta t}=\lim_{h \rightarrow 0} \frac{A(w_h)-A(w)}{h}.
$$
To define the $\frac{\delta }{\delta t}$-derivative analytically, let $\xi^\alpha$ be surface coordinates and 

$$
z^i=z^i(\xi, t)
$$
a parametric equation for $S_t$. The velocity object is 

$$
v^i(\xi, t)=\frac{\partial z^i(\xi, t)}{\partial t};
$$ 
and, its surface projection is

$$
v^\alpha=z_i^\alpha v^i.
$$
Then

$$
\frac{\delta A}{\delta t}=\frac{\partial A(\xi, t)}{\partial t}-v^\alpha \nabla_\alpha A.
$$
This may also be defined for typical tensors $T_{j\beta}^{i\alpha}$:

$$
\frac{\delta T_{j\beta}^{i\alpha}}{\delta t}=\frac{\partial T^i}{\partial t} - v^\gamma \nabla_\gamma T_{j\beta}^{i\alpha}+v^m \Gamma_{mk}^{i} T_{j\beta}^{k \alpha}-v^m \Gamma_{mj}^k T_{k\beta}^{i\alpha} + \nabla_\gamma v^\alpha T_{j \beta}^{i \gamma}-\nabla_\beta v^\gamma T_{j \gamma}^{i \alpha},
$$
where $\Gamma_{jk}^i$ are the Christoffel symbols.
To analyze the surface's speed, one defines the velocity of the surface along the normal $N$ as the projection of the radius vector $z$ onto the normal $N$

$$
C=\frac{\delta z  }{\delta t} \cdot N.
$$

Assume $\phi(S,t)$ is a surface restriction of a spatial field $f$, then the chain rule is:

$$
\frac{\delta \phi}{\delta t}=\frac{\partial f (z,t)}{\partial t}+C N^i \nabla_i f,
$$
with $N$ the surface normal;
and, the product rule:

$$
\frac{\delta (\phi \gamma)}{\delta t}=\phi \frac{\delta \gamma}{\delta t}+\frac{\delta \phi}{\delta t} \gamma.
$$
If $\Omega_t$ evolves and its boundary is $S_t$, 

\begin{equation*} \label{j}
\frac{d}{dt} \int_{\Omega_t} f d\Omega_t = \int_{\Omega_t} \frac{\partial f}{\partial t} d\Omega_t + \int_{S_t} C f dS_t,
\end{equation*}

\begin{equation} \label{-}
\frac{d}{dt} \int_{S_t} \phi dS_t = \int_{S_t} \frac{\delta \phi}{\delta t} dS_t - \int_{S_t} CB_\alpha^\alpha \phi dS_t,
\end{equation}
where $B_\beta^\alpha$ is the curvature tensor and its trace is the mean curvature
$B_\alpha^\alpha$ (in the tensor calculus convention, a circle of radius $R$ has $B_\alpha^\alpha=-\frac{1}{R}$).

In classical perturbation theory, the bulk equation is the standard Laplace eigenvalue equation:

$$
\nabla_i \nabla^i u +\lambda u = 0;
$$
the quantum mechanics sign convention is that the operator is positive definite and the bulk is also coupled with the Dirichlet boundary condition $u|_S=0$. Lastly, $u$ is normalized
$$
\int_\Omega u^2 d\Omega =1.
$$
To analyze the perturbed system, the bulk is differentiated in the $\frac{\partial}{\partial t}$ sense, the boundary condition in the $\frac{\delta}{\delta t}$ sense, and the normalization in the $\frac{d}{dt}$ sense

\begin{equation} \label{n}
\nabla_i \nabla^i \frac{\partial u}{\partial t}+\frac{d \lambda}{d t}u+\lambda \frac{\partial u}{\partial t} = 0,
\end{equation}
in $\Omega_t$;

\begin{equation} \label{2}
\frac{\partial u}{\partial t}+CN^i \nabla_i u = 0,
\end{equation}
on $S_t$;

\begin{equation} \label{9}
\int_{\Omega_t} u \frac{\partial u}{\partial t}d\Omega_t=0.
\end{equation}
To compute $\frac{d \lambda}{dt}$, one multiplies \eqref{n} by $u$, integrates over $\Omega_t$, and utilizes \eqref{2} \& \eqref{9}:

\begin{align*}
0&=\int_{\Omega_t} \Big(u \nabla_i \nabla^i \frac{\partial u}{\partial t}+\frac{d\lambda}{dt}u^2+\lambda \frac{\partial u}{\partial t}u \Big) d \Omega_t\\
&=-\int_{\Omega_t} \nabla_iu \nabla^i \frac{\partial u}{\partial t} d \Omega_t+\frac{d \lambda}{dt} \\
&=- \int_{S_t}N^i \frac{\partial u}{\partial t} \nabla_i u d S_t  +\frac{d \lambda}{dt}\\
&=  \int_{S_t}C(N^i \nabla_i u)^2 d S_t  +\frac{d \lambda}{dt}.
\end{align*}
Hence

$$
\frac{d \lambda}{dt}=-\int_{S_t}C(N^i \nabla_i u)^2 d S_t.
$$
The surface gradient $\nabla_\alpha u$ vanishes since $u|_{S_t}=0$, and therefore 

$$
\nabla^i u \nabla_i u= (N^i \nabla_i u)^2.
$$
Analogously, \eqref{-} and the tensor calculus formulas yield

\begin{align*}
\frac{d^2 \lambda}{dt^2}&=-\frac{d}{dt} \int_{S_t}C\nabla^i u \nabla_i u d S_t\\
&=-\int_{S_t} \frac{\delta C}{\delta t}(\nabla^i u \nabla_i u) d S_t -2\int_{S_t} C(\nabla^i \frac{\partial u}{\partial t}+C N^j \nabla_j \nabla^i u) \nabla_i udS_t\\
&+ \int_{S_t}C^2 B_\alpha^\alpha \nabla^i u \nabla_i u dS_t.
\end{align*}

Let $n \geq 3$ and $P \subset \mathbb{R} ^2$ be an $n$-gon generated by the set of vertices $\{A_1, A_2,
\ldots, A_n \} \subset  \mathbb{R}^2$ whose center of mass $O$ is taken to be the origin. For $i \in \{1,2,\ldots,n\}$, the $i$-th side length of $P$, denoted by $l_i:=A_iA_{i+1}$, is the length of the vector $\overrightarrow{A_iA_{i+1}}$ which connects $A_i$ to $A_{i+1}$, where $A_i=A_j$ if and only if $i=j$ (mod $n$); with this notation in mind, $\{r_i:=OA_i\}_{i=1}^n$ is the set of radii. Furthermore, $x_i$ is the angle between the vectors $\overrightarrow{OA_i}$ and $\overrightarrow{OA_{i+1}}$ and the set $\{x_i\}_{i=1}^n$ comprises the barycentric angles of $P$.

The circulant matrix method introduced in \cite{MR3327086} is based on the idea that a large class of polygons can be viewed as points in $\mathbb{R}^{2n}$ satisfying some constraints. One way of generating a large collection of polygons for investigating the P\'olya-Szeg\"o problem is by letting $\alpha>0$ $\&$ setting
\begin{equation*}
  \mathcal{M}(n, \alpha):=\Big \{ (x;r) \in  \mathbb{R} ^{2n}:\ x_i,r_i \geq 0,\ \eqref{eq: sum_x_i=2pi},\ \eqref{eq: alpha},\ \eqref{eq: centroid_cond} \ \mbox{hold} \Big\},
\end{equation*}
where
\begin{equation}\label{eq: sum_x_i=2pi}
  \sum \limits _{i=1}^{n} x_i=2 \pi,
\end{equation}
\begin{equation}\label{eq: alpha}
 \frac{1}{2} \sum \limits_{i=1}^n r_i r_{i+1} \sin x_i=\alpha.
\end{equation}
\begin{equation}\label{eq: centroid_cond}
  \begin{cases}
  \sum \limits _{i=1}^{n} r_i \cos \left(\sum \limits_{k=1}^{i-1} x_k \right)=0,\\
  \sum \limits _{i=1}^{n} r_i \sin \left(\sum \limits_{k=1}^{i-1} x_k \right)=0.
  \end{cases}
\end{equation}
$\mathcal{M}(n, \alpha)$ is a $(2n-4)$--dimensional polygonal manifold where each point $(x; r) \in \mathcal{M}(n, \alpha)$ represents a polygon centered at the origin with barycentric angles $x$, area $\alpha$, and radii $r$: a point $O$ is the barycenter of the set of vertices of $P$ means
\begin{equation*}
  \sum \limits _{i=1}^{n} \overrightarrow{OA_i}=0,
\end{equation*}
which is equivalent to saying that the projections of $\sum \limits _{i=1}^{n} \overrightarrow{OA_i}$ onto $\overrightarrow{OA_1}$ and $\overrightarrow{OA_1}^\perp$ vanish, therefore $(x;r)$ satisfies \eqref{eq: centroid_cond}. 

Furthermore, \eqref{eq: sum_x_i=2pi} is satisfied by all convex polygons (\& many nonconvex ones) and \eqref{eq: alpha} is encoding the given area. Note that the convex regular $n$-gon corresponds to the point $(x_*; r_*)=\left(\frac{2\pi}{n},\ldots,\frac{2\pi}{n};1,\ldots,1\right)$. 
Therefore, the variance of the interior angles, radii of $P$, and sides of $P$ are represented, respectively, by the quantities 

$$\sigma_a^2(P)=\sigma_a^2(x;r):=\frac{1}{n} \sum \limits_{i=1}^{n} x_i^2 - \frac{1}{n^2} \left(\sum \limits_{i=1}^{n}
x_i\right)^2,$$

$$\sigma_r^2(P)=\sigma_r^2(x;r):=\frac{1}{n} \sum \limits_{i=1}^{n} r_i^2 - \frac{1}{n^2} \left(\sum \limits_{i=1}^{n}
r_i\right)^2,$$

\begin{align*}
&\sigma_s^2(P)=\sigma_s^2(x;r):=\\
&\frac{1}{n} \sum \limits_{i=1}^{n} \left( r_{i+1}^2+r_i^2-2r_{i+1}r_i \cos x_i\right) - \frac{1}{n^2} \left(\sum \limits_{i=1}^{n}
\left( r_{i+1}^2+r_i^2-2r_{i+1}r_i \cos
x_i\right)^{1/2}\right)^2,
\end{align*}
\& in $(x;r)$ coordinates, the deficit is given by 

\begin{align*}
\delta(P)&=\delta(x; r):=\\
&\left(\sum \limits_{i=1}^{n} \left( r_{i+1}^2+r_i^2-2r_{i+1}r_i \cos
x_i\right)^{1/2}\right)^2-2n  \tan \frac{\pi}{n} \sum \limits_{i=1}^n r_i r_{i+1} \sin x_i.
\end{align*} 
In my approach to address the P\'olya-Szeg\"o problem, the stability of the polygonal isoperimetric inequality has a central importance. The inequality proved in \cite{MR3327086} when $P$ is a convex polygon and improved in \cite{MR3487241} compares the deficit with the variation $v(P)=\sigma_s^2(P)+\sigma_r^2(P)$: 

$$
v(P)+|P|\sigma_a^2(P) \le c_n \delta(P). 
$$

\section{The P\'olya-Szeg\"o problem; a formula for the principal frequency of a convex polygon; a 2006 conjecture of Antunes and Freitas; a solution to the sharp polygonal Faber-Krahn stability problem for triangles}
\subsection{Proof of Theorem \ref{a}}
\begin{proof}
Assume firstly $P\in \mathcal{M}(n, \alpha)$ is convex and without loss of generality $\alpha=|P_n|$ for the regular convex polygon inscribed in $\mathbb{S}^1$. Let $v_1, v_2, v_3$ denote three vertices taken consecutively clockwise. 
Set $T_0$ as the triangle generated via the vertices and $T_0^*$, the triangle generated via a Steiner symmetrization with respect to the line (without loss the x-axis) perpendicular to the line containing $v_1, v_3$ (without loss the y-axis) and which intersects the mid-point (without loss the origin) of the segment between $v_1, v_3$. Let $v_2^*$ denote the symmetrized vertex on the x-axis, $\xi=|v_2^*|$, $b=2|v_1|$, $t^*=|v_2-v_2^*|$, 
$$l_{1,2}=\{(x,y_+(x)), 0\le x \le \xi\}$$ 
the line segment connecting $v_1$ and $v_2^*$, 
$$l_{3,2}=\{(x,y_-(x)), 0\le x \le \xi\}$$ 
the line segment connecting $v_3$ and $v_2^*$, $P^*$ the polygon with $T_0$ replaced by $T_0^*$. The principal frequency is invariant via reflection, hence without loss up to a reflection $t^*\ge 0$.
Let $P_t$ denote the polygon in the evolution via Steiner symmetrization where $P_0=P^*$ and $P_{t^*}=P$. Note that the vertices of the triangle generating the symmetry are $v_1, v_2^t, v_3$, where $v_2^t=(\xi, t)$. Define the line-segment connecting $v_1$ and $v_2^t$ via $l_{1,2}(t)$ as in Figure \ref{e9}.

\begin{figure}[htbp] 
\centering
\includegraphics[width=.84 \textwidth]{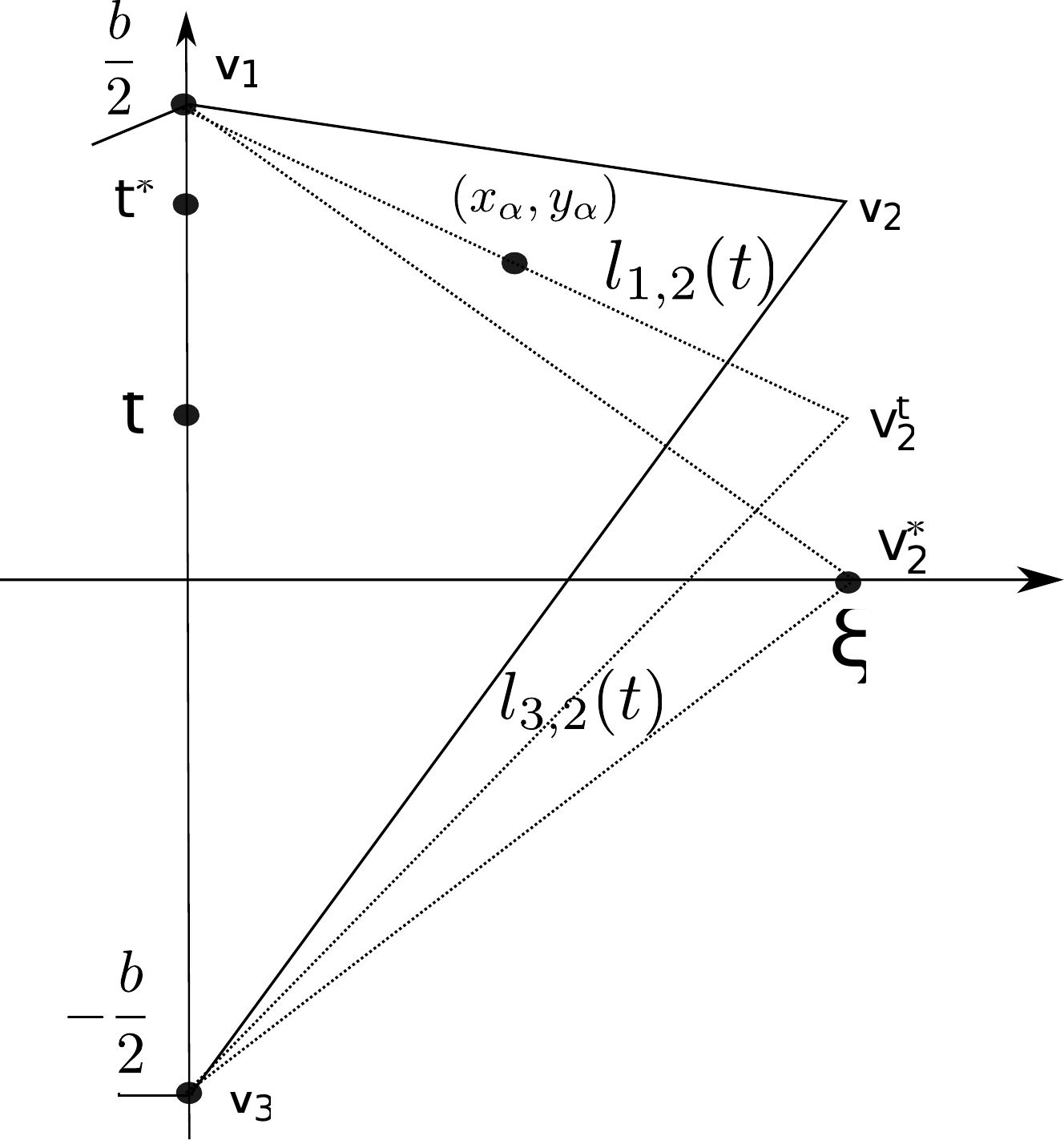}
\caption{}
\label{e9}
\end{figure}

\noindent The standard Laplace eigenvalue bulk equation in the interior of $P_t$ is 
\begin{equation*}
\nabla_i \nabla^i u+\lambda u =0;
\end{equation*}
the Dirichlet boundary condition is $u|_{S_t}=0$ and the normalization is 
$$
\int_{P_t} u^2 dP_t=1.
$$
The differentiated equations are 
$$
\nabla_i \nabla^i \partial_t u+\frac{d \lambda}{dt}u+\lambda \partial_t u = 0
$$ 
in the interior of $P_t$;
\begin{equation} \label{e}
\partial_t u + C N^i \nabla_i u = 0
\end{equation}
at the boundary $S_t$ where $N$ is the outer normal and $C$ the surface velocity; suppose $w_t \in l_{1,2}(t)\setminus\{v_1, v_2^t\}$, then one may extend the boundary equation on a neighborhood $Q_t$ of $w_t$ since $N$ is constant on $l_{1,2}(t)$ (the surface velocity can be extended since the particles move via the symmetrization). Thus 
\begin{equation} \label{e17}
\partial_t \nabla u+\nabla C(N\cdot \nabla u)+C D^2 u N =0
\end{equation}
in $Q_t$ thanks to $\big(\partial_{x_i} N^i\big) \partial_{x_j} u=0$; the normalization yields
$$
\int_{P_t} u \partial_t u dP_t =0.
$$
Therefore, integration by parts, the divergence theorem, and the equations imply
\begin{equation} \label{vw}
\frac{d \lambda}{dt}=-\int_{S_t} C(N^i \nabla_i u)^2 dS_t;
\end{equation}
\begin{align}
&\frac{d^2 \lambda}{dt^2} \label{q4e}\\ \notag
&=-\frac{d}{dt}\int_{S_t} C(N^i \nabla_i u)^2 dS_t\\\notag
&=-\int_{S_t} \frac{\delta C}{\delta t}(\nabla^i u \nabla_i u) d S_t-2 \int_{S_t} C(\nabla^i \partial_t u+C N^j \nabla_j \nabla^i u) \nabla_i udS_t\\\notag
&\hskip .981in + \int_{S_t}C^2 B_\alpha^\alpha \nabla^i u \nabla_i u dS_t.
\end{align}     
 Utilizing \eqref{e17}
 \begin{align*}
& \int_{S_t} C(\nabla^i \partial_t u+C N^j \nabla_j \nabla^i u) \nabla_i udS_t\\
&=\int_{S_t} C(- (N\cdot \nabla u)\nabla C \cdot \nabla u-CD^2uN\cdot \nabla u+C\nabla u \cdot \nabla \partial_N u) dS_t;
 \end{align*}    
 hence, as $N$ is constant in a neighborhood of an $interior-to-l_{1,2}(t)$ boundary point, 
 $$
 \nabla u \cdot \nabla \partial_N u=D^2uN \cdot \nabla u; $$   therefore, since $N$ is parallel to $\nabla u$ \& constant along the segment
 \begin{align*}
\frac{d^2 \lambda}{dt^2}&=-\int_{S_t} \frac{\delta C}{\delta t}(\nabla^i u \nabla_i u) d S_t-
2 \int_{S_t} C(-\nabla C \cdot N)|\nabla u|^2dS_t\\
&+ \int_{S_t}C^2 B_\alpha^\alpha \nabla^i u \nabla_i u dS_t;
&\end{align*} 
since the symmetrization induces solely the triangle to change, $C$ is vanishing on $S_t \setminus (l_{1,2}(t) \cup l_{3,2}(t))$. \\

{\bf The case: $y_{+}$ (the upper line-segment)} \\ 
Define
$$
 \begin{bmatrix}
x_\alpha
\\
y_\alpha
\end{bmatrix}=
 \begin{bmatrix}
\alpha
\\
-\frac{b}{2\xi}\alpha + \frac{b}{2}
\end{bmatrix}$$ 
$$
z= \begin{bmatrix}
z^x(\alpha, t)
\\
z^y(\alpha,t)
\end{bmatrix}=
 \begin{bmatrix}
\alpha
\\
\frac{(t-\frac{b}{2})}{\xi}\alpha + \frac{b}{2}
\end{bmatrix}$$

$$
N= \begin{bmatrix}
\frac{b}{2}-t
\\
\xi
\end{bmatrix}\frac{1}{\sqrt{\xi^2+(\frac{b}{2}-t)^2}}
$$
$$
\partial_t \begin{bmatrix}
z^x(\alpha, t)
\\
z^y(\alpha,t)
\end{bmatrix}=
 \begin{bmatrix}
0
\\
\frac{\alpha}{\xi}
\end{bmatrix}$$

$$
\partial_\alpha \begin{bmatrix}
z^x(\alpha, t)
\\
z^y(\alpha,t)
\end{bmatrix}=
 \begin{bmatrix}
1
\\
\frac{t-\frac{b}{2}}{\xi}
\end{bmatrix}$$

$$v^\alpha=\alpha \frac{(t-\frac{b}{2})}{\xi^2+(t-\frac{b}{2})^2}$$

\begin{align*}
\frac{\delta z}{\delta t}&=\partial_tz-v^\alpha \partial_\alpha z\\
&=\frac{\alpha}{\xi^2+(t-\frac{b}{2})^2} \begin{bmatrix}
\frac{b}{2}-t
\\
\xi
\end{bmatrix}
\end{align*}

 \begin{equation} \label{on}
 C=\frac{\delta z}{\delta t} \cdot N=
\frac{\alpha}{\xi^2+(t-\frac{b}{2})^2} \begin{bmatrix}
\frac{b}{2}-t
\\
\xi
\end{bmatrix}
\cdot
\begin{bmatrix}
\frac{b}{2}-t
\\
\xi
\end{bmatrix}\frac{1}{\sqrt{\xi^2+(\frac{b}{2}-t)^2}}=\frac{\alpha}{\sqrt{\xi^2+(\frac{b}{2}-t)^2}}
 \end{equation}
 
 $$ \partial_tC
=\frac{\alpha (\frac{b}{2}-t)}{(\xi^2+(\frac{b}{2}-t)^2)^{3/2}}
$$
 $$ \partial_\alpha C=\frac{1}{\sqrt{\xi^2+(t-\frac{b}{2})^2}}.$$

In particular, 
\begin{align*}
\frac{\delta C}{\delta t}&=\partial_tC-v^\alpha \nabla_\alpha C\\
&=\frac{\alpha (\frac{b}{2}-t)}{(\xi^2+(\frac{b}{2}-t)^2)^{3/2}}-\big( \alpha \frac{(t-\frac{b}{2})}{\xi^2+(t-\frac{b}{2})^2}\big) \big( \frac{1}{\sqrt{\xi^2+(t-\frac{b}{2})^2}}\big)\\
&=2\frac{\alpha (\frac{b}{2}-t)}{(\xi^2+(\frac{b}{2}-t)^2)^{3/2}};
\end{align*}
in order to compute $\nabla C$, set $x=\alpha$, $y=\frac{t-\frac{b}{2}}{\xi}x+\frac{b}{2}$:

$$
C=\frac{x^2}{\xi\sqrt{x^2+(y-\frac{b}{2})^2}}
$$

 $$
 \nabla C=
 \begin{bmatrix}
\frac{x^3+2x(y-\frac{b}{2})^2}{\xi(x^2+(y-\frac{b}{2})^2)^{\frac{3}{2}}}
\\
\frac{x^2(\frac{b}{2}-y)}{\xi (x^2+(y-\frac{b}{2})^2)^{\frac{3}{2}}}
\end{bmatrix}
$$

\begin{align*}
\nabla C \cdot N&= \begin{bmatrix}
\frac{x^3+2x(y-\frac{b}{2})^2}{\xi(x^2+(y-\frac{b}{2})^2)^{\frac{3}{2}}}
\\
\frac{x^2(\frac{b}{2}-y)}{\xi (x^2+(y-\frac{b}{2})^2)^{\frac{3}{2}}}
\end{bmatrix} \cdot \begin{bmatrix}
\frac{b}{2}-t
\\
\xi
\end{bmatrix}\frac{1}{\sqrt{\xi^2+(\frac{b}{2}-t)^2}}\\
&= \frac{2(\frac{b}{2}-t)}{\xi^2+(t-\frac{b}{2})^2};
\end{align*}
since $B_\alpha^\alpha$ has a Dirac mass at vertices of non-degenerate sides, by a comparison principle argument $|\nabla u(v_2^*)|=0$: one can compare $u$ with the solution of the corresponding equation in a disk intersecting the vertex $v_2^*$ and containing the polygon; moreover, $C(v_1)=0$. This yields
 \begin{align} \label{a*36}
&\Big(-\int_{l_{1,2}(t)} \frac{\delta C}{\delta t}(\nabla^i u \nabla_i u) dH-
2 \int_{l_{1,2}(t)} C(-\nabla C \cdot N)|\nabla u|^2dH\\ \notag
&+ \int_{l_{1,2}(t)}C^2 B_\alpha^\alpha \nabla^i u \nabla_i u dH \Big)\\ \notag
&=-\frac{2(\frac{b}{2}-t)}{\xi(\xi^2+(\frac{b}{2}-t)^2)}\int_{0}^\xi \alpha |\nabla u|^2 d\alpha+ \frac{4(\frac{b}{2}-t)}{\xi(\xi^2+(t-\frac{b}{2})^2)} \int_0^\xi\alpha|\nabla u|^2 d\alpha\\\notag
&=\frac{2(\frac{b}{2}-t)}{\xi(\xi^2+(t-\frac{b}{2})^2)}\int_0^\xi \alpha |\nabla u|^2 d \alpha
\notag
\end{align} 

{\bf The case: $y_{-}$ (the lower line-segment)} \\ 

$$
z= \begin{bmatrix}
z^x(\alpha, t)
\\
z^y(\alpha,t)
\end{bmatrix}=
 \begin{bmatrix}
\alpha
\\
\frac{(t +\frac{b}{2})}{\xi}\alpha - \frac{b}{2}
\end{bmatrix}$$

$$
N= \begin{bmatrix}
\frac{b}{2}+t
\\
-\xi
\end{bmatrix}\frac{1}{\sqrt{\xi^2+(\frac{b}{2} +t)^2}}
$$

$$
\partial_\alpha \begin{bmatrix}
z^x(\alpha, t)
\\
z^y(\alpha,t)
\end{bmatrix}=
 \begin{bmatrix}
1
\\
\frac{t +\frac{b}{2}}{\xi}
\end{bmatrix}$$

$$
\partial_t \begin{bmatrix}
z^x(\alpha, t)
\\
z^y(\alpha,t)
\end{bmatrix}=
 \begin{bmatrix}
0
\\
\frac{\alpha}{\xi}
\end{bmatrix}$$

$$v^\alpha=\alpha \frac{(t+\frac{b}{2})}{\xi^2+(t+\frac{b}{2})^2}$$

\begin{align*}
\frac{\delta z}{\delta t}&=\partial_tz-v^\alpha \partial_\alpha z\\
&=\frac{-\alpha}{\xi^2+(t+\frac{b}{2})^2} \begin{bmatrix}
\frac{b}{2}+t
\\
-\xi
\end{bmatrix}
\end{align*}

 \begin{equation} \label{zn}
 C=\frac{\delta z}{\delta t} \cdot N=
\frac{-\alpha}{\xi^2+(t+\frac{b}{2})^2} \begin{bmatrix}
\frac{b}{2}+t
\\
-\xi
\end{bmatrix}
\cdot
\begin{bmatrix}
\frac{b}{2}+t
\\
-\xi
\end{bmatrix}\frac{1}{\sqrt{\xi^2+(\frac{b}{2}+t)^2}}=\frac{-\alpha}{\sqrt{\xi^2+(\frac{b}{2}+t)^2}}
\end{equation}
 
 $$ \partial_tC
=\frac{\alpha (\frac{b}{2}+t)}{(\xi^2+(\frac{b}{2}+t)^2)^{3/2}}
$$
 $$ \partial_\alpha C=\frac{-1}{\sqrt{\xi^2+(t+\frac{b}{2})^2}}.$$

In particular, 
\begin{align*}
\frac{\delta C}{\delta t}&=\partial_tC-v^\alpha \nabla_\alpha C\\
&=\frac{\alpha (\frac{b}{2}+t)}{(\xi^2+(\frac{b}{2}+t)^2)^{3/2}}-\big( \alpha \frac{(t+\frac{b}{2})}{\xi^2+(t+\frac{b}{2})^2}\big) \big( \frac{-1}{\sqrt{\xi^2+(t+\frac{b}{2})^2}}\big)\\
&= 2\frac{\alpha (\frac{b}{2}+t)}{(\xi^2+(\frac{b}{2}+t)^2)^{3/2}};
\end{align*}
in order to compute $\nabla C$, set $x=\alpha$, $y=\frac{t+\frac{b}{2}}{\xi}x-\frac{b}{2}$:

$$
C=\frac{-x^2}{\xi\sqrt{x^2+(y+\frac{b}{2})^2}}
$$

 $$
 \nabla C=
 \begin{bmatrix}
-\frac{x^3+2x(y+\frac{b}{2})^2}{\xi(x^2+(y+\frac{b}{2})^2)^{\frac{3}{2}}}
\\
\frac{x^2(\frac{b}{2}+y)}{\xi (x^2+(y+\frac{b}{2})^2)^{\frac{3}{2}}}
\end{bmatrix}
$$

\begin{align*}
\nabla C \cdot N&=\begin{bmatrix}
-\frac{x^3+2x(y+\frac{b}{2})^2}{\xi(x^2+(y+\frac{b}{2})^2)^{\frac{3}{2}}}
\\
\frac{x^2(\frac{b}{2}+y)}{\xi (x^2+(y+\frac{b}{2})^2)^{\frac{3}{2}}}
\end{bmatrix} \cdot
\begin{bmatrix}
\frac{b}{2}+t
\\
-\xi
\end{bmatrix}\frac{1}{\sqrt{\xi^2+(\frac{b}{2}+t)^2}}\\
&=\frac{-2(\frac{b}{2}+t)}{\xi^2+(t+\frac{b}{2})^2};
\end{align*}
therefore as above since $B_\alpha^\alpha$ has a Dirac mass at the vertex, by comparing $u$ with the solution of the corresponding equation in a disk intersecting the vertex $v_2^*$ and containing the polygon, $|\nabla u(v_2^*)|=0$ (moreover, $C(v_3)=0$). This yields
 \begin{align} \label{a*53}
&\Big(-\int_{l_{1,2}(t)} \frac{\delta C}{\delta t}(\nabla^i u \nabla_i u) dH-
2 \int_{l_{1,2}(t)} C(-\nabla C \cdot N)|\nabla u|^2dH \notag\\
&+ \int_{l_{1,2}(t)}C^2 B_\alpha^\alpha \nabla^i u \nabla_i u dH \Big)\\\notag
&=-\frac{2(\frac{b}{2}+t)}{\xi(\xi^2+(\frac{b}{2}+t)^2)}\int_{0}^\xi \alpha |\nabla u|^2 d\alpha+ \frac{4(\frac{b}{2}+t)}{\xi(\xi^2+(\frac{b}{2}+t)^2)} \int_0^\xi\alpha|\nabla u|^2 d\alpha\\\notag
&=\frac{2(\frac{b}{2}+t)}{\xi(\xi^2+(t+\frac{b}{2})^2)}\int_0^\xi \alpha |\nabla u|^2 d \alpha.
\notag
\end{align} 
Hence \eqref{on}, \eqref{a*36}, \eqref{zn},  \& \eqref{a*53} imply

$$
\frac{d \lambda}{dt}=-\int_{S_t} C(N^i \nabla_i u)^2 dS_t=\int_0^\xi \frac{\alpha}{\xi}|\nabla u(\alpha, y_{-})|^2 d\alpha-\int_0^\xi \frac{\alpha}{\xi} |\nabla u(\alpha, y_+)|^2 d\alpha;
$$

$$
\frac{d^2 \lambda}{d t^2}=\frac{2(\frac{b}{2}-t)}{\xi(\xi^2+(t-\frac{b}{2})^2)}\int_0^\xi \alpha |\nabla u(\alpha, y_+)|^2 d \alpha+\frac{2(\frac{b}{2}+t)}{\xi(\xi^2+(t+\frac{b}{2})^2)}\int_0^\xi \alpha |\nabla u(\alpha, y_{-})|^2 d \alpha;
$$
and,
$$
\lambda(P)=\lambda(P^*)+\frac{d \lambda}{d t}\Big|_{t=0}t^*+\frac{d^2 \lambda}{d t^2}\Big|_{t=0}\frac{(t^*)^2}{2}+o((t^*)^2).
$$
Also if
\begin{equation*} \label{ae}
\int_0^\xi \alpha|\nabla u(\alpha, y_{-})|^2 d\alpha \ge \int_0^\xi \alpha |\nabla u(\alpha, y_+)|^2 d\alpha,
\end{equation*}
\begin{align*}
&\frac{d \lambda}{dt}\Big|_{t=0}\ge 0;
\end{align*}
next, if
\begin{equation*} \label{ae}
\int_0^\xi \alpha|\nabla u(\alpha, y_{-})|^2 d\alpha < \int_0^\xi \alpha |\nabla u(\alpha, y_+)|^2 d\alpha
\end{equation*}
set $P^{2}=P^*$, $P^3=(P^{*})^{*}$, where the vertices $v_2^*, v_3, v_4$ are taken from $P^*$ ($v_4$ is the next clockwise vertex relative to $v_3$) \& $P^3$ is obtained via the triangle symmetrization relative to $P^*$ as the initial n-gon (Figure \ref{a5w}). 

\begin{figure}[htbp] 
\centering
\includegraphics[width=.82 \textwidth]{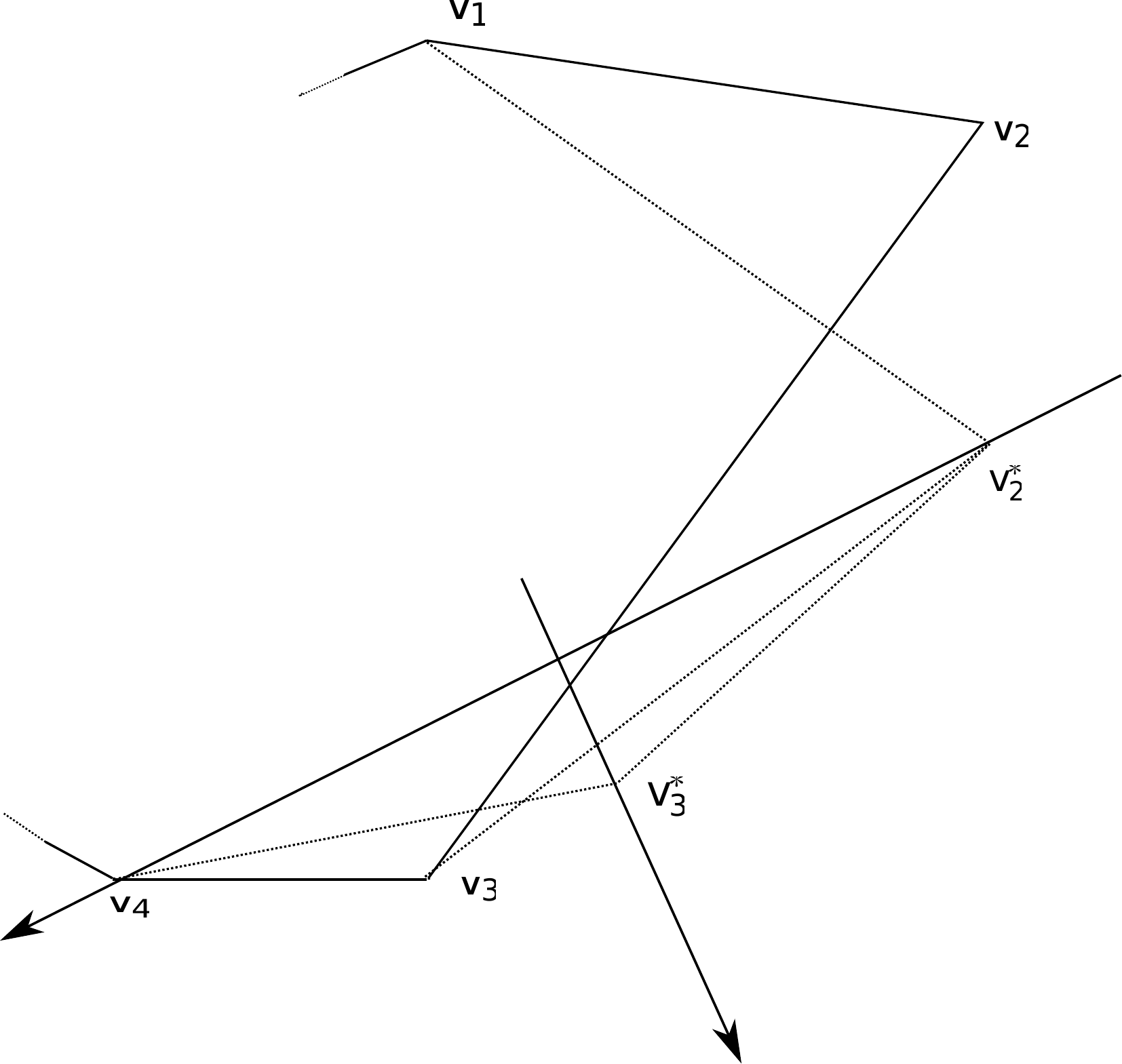}
\caption{}
\label{a5w}
\end{figure}

\noindent Define $P^{k+1}:=(P^{k})^{*}$, $t_{k}=|v_{k+1}^*-v_{k+1}|$ with $P^1=P$ and for constants $a_n$ (specified in the latter), assuming $w_i$ are the vertices of $P_n,$

\begin{align}
\label{o2} &A_{n}:=\\ \notag
&\Big\{\text{vertices}(P) \in \bigotimes_{i=1}^{n} B_{a_n}(w_i): \sum_{j \ge k} t_j^2\le \alpha t_{k-1}^2 \\
& \text{ whenever $t_{k-1}>0$}, P^\infty \neq P^i \text{ if $i \in \mathbb{N}$}, P^\infty = P_n\Big\}\notag
\end{align}
(in the general case, $P^\infty$ is equilateral, not necessarily cyclical). Now set 
\begin{equation} \label{la}
\phi(t)=\sqrt{(\frac{b}{2}-t)^2+\xi^2}+\sqrt{(\frac{b}{2}+t)^2+\xi^2}-2\sqrt{(\frac{b}{2})^2+\xi^2}
\end{equation}
so that 

$$
\phi'(0)=0,
$$

\begin{align*}
\phi''(t)&=\frac{1}{\sqrt{(\frac{b}{2}-t)^2+\xi^2}}+\frac{1}{\sqrt{(\frac{b}{2}+t)^2+\xi^2}}-\frac{(\frac{b}{2}-t)^2}{((\frac{b}{2}-t)^2+\xi^2)^{3/2}}\\
&-\frac{(\frac{b}{2}+t)^2}{((\frac{b}{2}+t)^2+\xi^2)^{3/2}}\\
& \approx \frac{1}{\sqrt{2}}\xi^{\frac{1}{2}}
\end{align*}
when $t \approx 0$; 
the above yields
\begin{align}
\label{ai@}
L(P^k)-L(P^{k+1}) &= \sqrt{(\frac{b}{2}-t_{k})^2+\xi^2}+\sqrt{(\frac{b}{2}+t_k)^2+\xi^2}-2\sqrt{(\frac{b}{2})^2+\xi^2} \\ \notag
&\approx \frac{1}{2\sqrt{2}}\xi^{\frac{1}{2}} t_k^2.                     
\end{align}
Observe also that

$$
\xi(n)=1-\cos(2\pi/n) \approx \xi_k=\xi 
$$
$$
b(n)=2\sin(2\pi/n) \approx b_k=b;
$$
In particular, supposing $t_1, t_2, \ldots, t_k >0$, stability implies \cite{MR3327086, MR3487241, MR3550852}

$$
\alpha(n) |R_k(P^k) \Delta P_n|^2 \le \delta(P^k)
$$
 with $\inf_n \alpha(n)>0$ and $R_k$ a rigid motion;  whence  supposing $t_1, t_2, \ldots, t_k >0$

\begin{align}
\label{ozm}
\alpha(n) |R_k(P^k) \Delta P_n|^2 \le \delta(P^k)&=L^2(P^k)-L^2(P_n)\\ \notag
&=\Big(L(P^k)+L(P_n)\Big) \Big(L(P^k)-L(P_n) \Big)\\ \notag
&=\Big(L(P^k)+L(P_n)\Big) \Big(\sum_{j=k}^\infty (L(P^j)-L(P^{j +1})) \Big)\\ \notag
&\le \bar a \xi^{\frac{1}{2}}(n) \Big(L(P^k)+L(P_n)\Big)  \sum_{j \ge k} t_j^2 \\ \notag
&\le   \tilde{a}\xi^{\frac{1}{2}}(n) \sum_{j \ge k} t_j^2\\ \notag
&\le r_a t_{k-1}^2, \notag
\end{align}
where $r_a>0$ is uniform in $k$ \& $n$ if $n$ is large. 
In particular, because 
$$\inf_n \alpha(n)>0,$$ when assuming $t_1, t_2, \ldots, t_k >0$,

$$|R_k(P^k) \Delta P^\infty| \le r t_{k-1}.$$
In general, for $k \in \mathbb{N}$ s.t.  $t_k \neq 0$, either: (a) $t_j= 0$ for $j \ge k+1$; (b)  $t_{k+1}\neq 0$; (c) $t_{k+l}\neq 0$ for an $l \ge 2$ and $t_{e}=0$ for $e \in \{k+1, \ldots, k+l-1 \}$ and in that context $s_{k+1}=\ldots=s_{k+l-1}=s_{k+l}   \neq s_{k+l+1}$ as in Figure \ref{3ai}; also, if $t_k=0$ for all $k \in \mathbb{N}$, $P=P^\infty$. Therefore, letting $t_{j_1}, t_{j_2}, \ldots, t_{j_k}, \ldots$ be the non-zero $t_k$, assuming $\{t_{j_k}\}$ is an infinite sequence (assuming the sequence is finite implies the convergence in finitely many iterations),

\begin{equation} \label{t_{i-1}}
|R_{j_k}(P^{j_k}) \Delta P_n| \le r t_{j_{k-1}}.
\end{equation}

\begin{figure}[htbp] 
\centering
\includegraphics[width=.76 \textwidth]{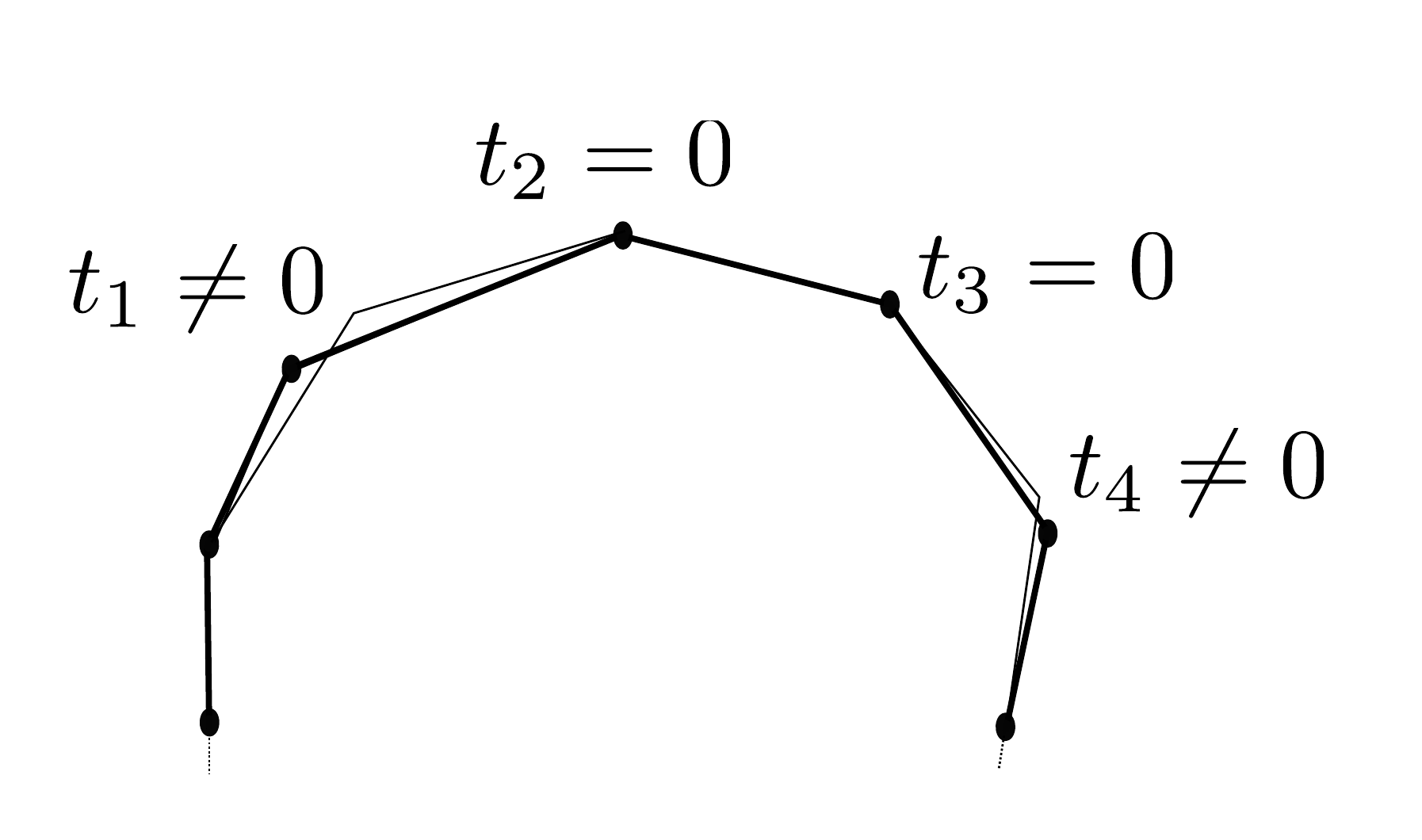}
\caption{}
\label{3ai}
\end{figure}
Moreover,

\begin{align*}
\lambda(P)&=\lambda(P^2)+\frac{d \lambda(P^2)}{d t}\Big|_{t=0}t_1+\frac{d^2 \lambda(P^2)}{d t^2}\Big|_{t=0}\frac{t_1^2}{2}+o(t_1^2)\\
&\ge \lambda(P^3)+\frac{d \lambda(P^3)}{dt}\Big|_{t=0}t_2+\frac{d^2 \lambda(P^3)}{dt^2}\Big|_{t=0}\frac{t_2^2}{7}+\frac{d \lambda(P^2)}{dt}\Big|_{t=0}t_1+\frac{d^2 \lambda(P^2)}{dt^2}\Big|_{t=0}\frac{t_1^2}{7}\\
&\ge \lambda(P^\infty)+\sum_{k=2}^\infty  \frac{d \lambda(P^{k})}{dt}\Big|_{t=0}t_{k-1}+\sum_{k=2}^\infty  \frac{d^2 \lambda(P^k)}{dt^2}\Big|_{t=0}\frac{t_{k-1}^2}{7},
\end{align*} 
since the second derivative for $n$ fixed is bounded from below uniformly in $k$ when $P \approx P_n$ via 

$$
\xi(n)=1-\cos(2\pi/n) \approx \xi_k, 
$$

$$
b(n)=2\sin(2\pi/n) \approx b_k,
$$
\eqref{lower6}, where $\xi_k, b_k$ correspond to the fixed n-gon $P^{k}$ \& $t_{k-1}$ is small: if $t_1=t$ is small, one may absorb $o(t_{k-1}^2)$ in the quadratic up to a constant 

$$
t_{k-1}^2 \le t_1^2+\sum_{j \ge 2} t_j^2 \le  (1+\alpha)t_1^2.
$$
Next, since the (simple) radial derivative of the eigenfunction of $B_1$ is

$$
\frac{d u_S}{dr}\Big|_{r=1}=\frac{-p}{\sqrt{\pi}},
$$
with $J_m(r)$ the Bessel function of order $m$, $p$ the first root of $J_0$; 
define 

$$l_{i}:=\{(x,y_+(x)), \frac{\xi}{5}\le x \le \frac{\xi}{4}\}.$$ 
\begin{figure}[htbp] 
\centering
\includegraphics[width=.73 \textwidth]{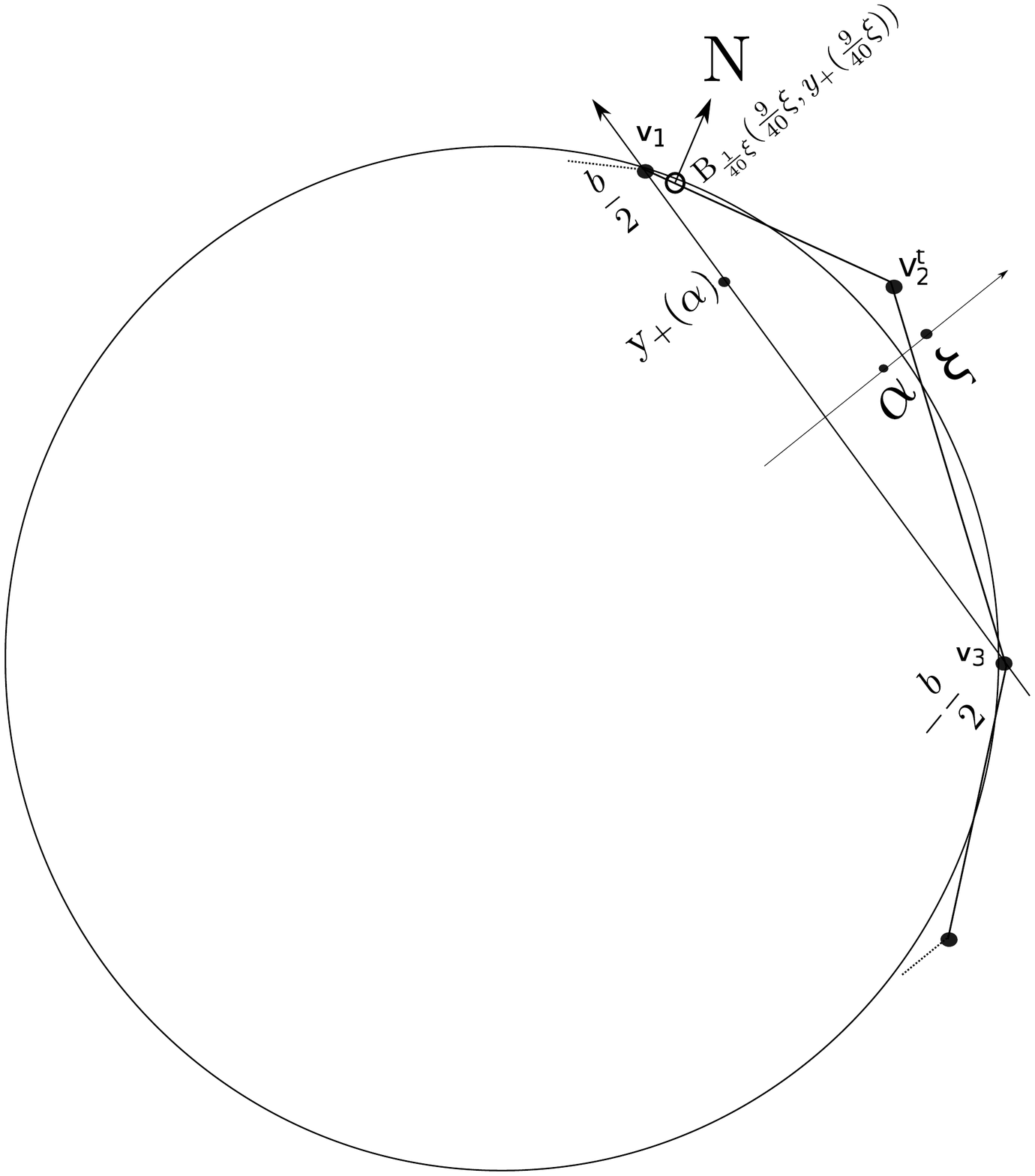}
\caption{}
\label{5ol}
\end{figure}
The eigenfunctions are analytic away from the vertices, therefore reflecting $u_n$, the eigenfunction of $P$ ($P \approx P_n$, thus one may consider any $P^k$),  
$$\Delta u_n=-\lambda_n u_n \hskip .1in \text{in } B_{j\xi}((9/40)\xi, y_{+}((9/40)\xi)),$$
where $j=\frac{1}{40}$ (Figure \ref{5ol});
if $\phi_n:=u_n-u_S$, where $u_S$ is the eigenfunction of $B_1$, 

$$
\Delta \phi_n = \lambda u_S -\lambda_n u_n \hskip .1in \text{in } B_{j\xi}((9/40)\xi, y_{+}((9/40)\xi));
$$
therefore, consider the Newtonian potential of $ \lambda u_S -\lambda_n u_n$: 

$$w_n:=N * ( \lambda u_S -\lambda_n u_n),$$ then 
$$
\Delta w_n=\lambda u_S -\lambda_n u_n \hskip .1in \text{in } B_{j\xi}((9/40)\xi, y_{+}((9/40)\xi))
$$
and hence $q_n=\phi_n-w_n$ is harmonic in $B_{j\xi}((9/40)\xi, y_{+}((9/40)\xi))$. Observe that \cite[Lemma 4.1]{gilbarg:01} implies 

\begin{equation} \label{z1}
\sup_{B_{j\xi}((9/40)\xi, y_{+}((9/40)\xi))} |D w_n| \rightarrow 0,
\end{equation}
as $n \rightarrow \infty$, because $\xi \rightarrow 0,$
$$\lambda_n \rightarrow \lambda$$
due to Chenais' theorem \cite[Theorem 2.3.18]{MR2251558}, $u_S|_{\partial B_1}=u_n|_{\partial P}=0$, and $$\lambda u_S -\lambda_n u_n \rightarrow 0.$$ 
Hence one obtains that the normal derivative  

$$\partial_N w_n((9/40)\xi, y_{+}((9/40)\xi)) \rightarrow 0$$ 
as $n\rightarrow \infty$. If $aN$ is a dilation along the outer normal relative to $l_{i}$ at $((9/40)\xi, y_{+}((9/40)\xi))$, $a<0$, $|a|$ small, then $aN \in P$ and as $P \rightarrow B_1$, if $n$ is large, up to a translation of $B_1$, $aN \in B_1$. Hence, assuming $\bar a>0$ is small, $q_n$ is well-defined and harmonic in a rectangle (Figure \ref{9z})
\begin{figure}[htbp] 
\centering
\includegraphics[width=.89 \textwidth]{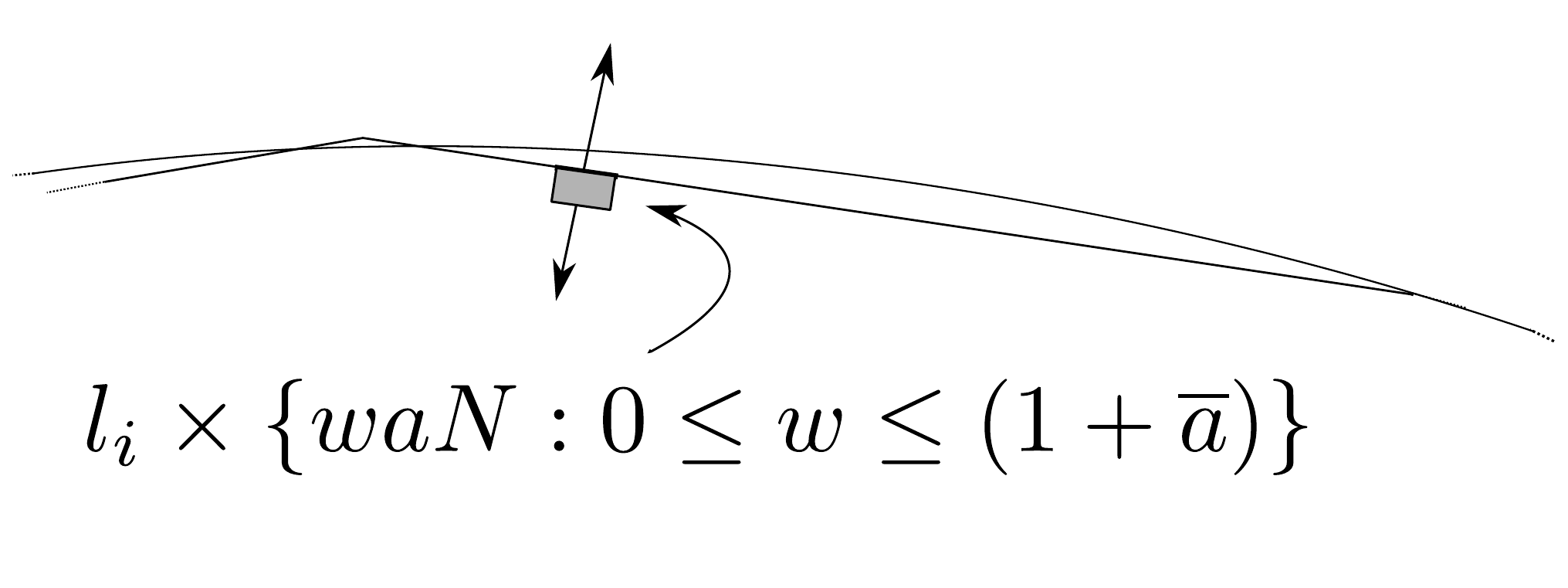}
\caption{}
\label{9z}
\end{figure}
$$l_{i} \cap B_{j\xi}((9/40)\xi, y_{+}((9/40)\xi)) \times \{ waN: 0\le w \le (1+\bar a)\} \subset P_n.$$ 
Since $\lambda_n \rightarrow \lambda$, $aN \in B_1 \cap P_n$, 

$$(\lambda u_S -\lambda_n u_n)(aN) \rightarrow 0.$$ Also, $u_n|_{\partial P}=0$ implies

$$(\lambda u_S -\lambda_n u_n)((9/40)\xi, y_{+}((9/40)\xi))=\lambda u_S ((9/40)\xi, y_{+}((9/40)\xi)) \rightarrow 0$$
because $u_S|_{\partial B_1}=0$; thus it follows that 
$$
q_n(((9/40)\xi, y_{+}((9/40)\xi))+aN)\rightarrow 0
$$
$$
q_n(((9/40)\xi, y_{+}((9/40)\xi)) \rightarrow 0;
$$
furthermore,
 \begin{align}\notag
 &\int_{0}^{|a|}|\partial_N q_n(((9/40)\xi, y_{+}((9/40)\xi))-wN)|^2 dw \label{q2o}\\ \notag
 &=q_n(((9/40)\xi, y_{+}((9/40)\xi))+aN) \partial_N q_n(((9/40)\xi, y_{+}((9/40)\xi))+aN)\\\notag
 &-q_n((9/40)\xi, y_{+}((9/40)\xi)) \partial_N q_n((9/40)\xi, y_{+}((9/40)\xi))\\\notag
 &-\int_0^{|a|} \partial_{NN}q_n(((9/40)\xi, y_{+}((9/40)\xi))-wN) q_n(((9/40)\xi, y_{+}((9/40)\xi))-wN) dw\\
 \end{align}
 and the $W^{2, p}$ estimates (the Calderon-Zygmund theory is efficiently stated in \cite{MR2999297})

\begin{align*}
&||D^2 q_n||_{L^p(B_{a_*r}(((9/40)\xi, y_{+}((9/40)\xi))-wN)} \le C_{2,p}\big(||\phi_n-w_n||_{L^1(B_{ar}(((9/40)\xi, y_{+}((9/40)\xi))-wN)}\big)\\
&=C_{2,p}\big(||(u_n-u_S)-N * ( \lambda u_S -\lambda_n u_n)||_{L^1(B_{ar}(((9/40)\xi, y_{+}((9/40)\xi))-wN)}\big),
\end{align*}
with $a_*<a$ sufficiently small imply 

\begin{equation}  \label{eal}
 |\partial_N q_n(((9/40)\xi, y_{+}((9/40)\xi))+aN)| + |\partial_N q_n((9/40)\xi, y_{+}((9/40)\xi))|\le \alpha:
\end{equation}
to simplify the notation, let $N=e_1$, then via Taylor's theorem
$$q_n'(x)=\frac{1}{2h} [q_n(x+2h)-q_n(x)]-h q_n''(\theta)$$
with $\theta \in (x,x+2h)$. This then implies

\begin{equation*} \label{ea6}
|q_n'(x)| \le h ||q_n''||_{L^\infty}+\frac{||q_n||_{L^\infty}}{h};
\end{equation*}
therefore optimizing in $h>0$ implies

\begin{equation} \label{p4m}
||q_n'||_{L^\infty} \le 2 \sqrt{||q_n''||_{L^\infty} ||q_n||_{L^\infty}}.
\end{equation}
In particular, to obtain \eqref{eal}, it is sufficient to obtain an $L^\infty$ bound on $ \partial_{NN} q_n$ close to $((9/40)\xi, y_{+}((9/40)\xi))+aN$ \& $((9/40)\xi, y_{+}((9/40)\xi))$. This follows from  $W^{2, p}$ and BMO estimates and the Lebesgue differentiation theorem (the argument is below). 
Note

\begin{align*}
&\int_0^{|a|} \partial_{NN}q_n(((9/40)\xi, y_{+}((9/40)\xi))-wN) q_n(((9/40)\xi, y_{+}((9/40)\xi))-wN) dw \\
&\le \Big(\int_0^{|a|} | \partial_{NN}q_n(((9/40)\xi, y_{+}((9/40)\xi))-wN)|^2 dw\Big)^{\frac{1}{2}}\times \\
&\Big(\int_0^{|a|} |q_n(((9/40)\xi, y_{+}((9/40)\xi))-wN)|^2 dw\Big)^{\frac{1}{2}} \rightarrow 0.
\end{align*}
One therefore obtains from \eqref{q2o} the convergence 

\begin{equation} \label{oy2}
\int_{0}^{|a|}|\partial_N q_n(((9/40)\xi, y_{+}((9/40)\xi))-wN)|^2 dw \rightarrow 0.
\end{equation}
Suppose 

\begin{align}
 &\limsup_{n \rightarrow \infty} |\partial_N q_n((9/40)\xi, y_{+}((9/40)\xi))|=\lim_{l \rightarrow \infty} |\partial_N q_{n_l}((9/40)\xi, y_{+}((9/40)\xi))| \label{h}\\\notag
 &>0.
\end{align}
Hence, \eqref{oy2} with an application of Fatou's lemma implies

$$\int_{0}^{|a|}\lim_{l\rightarrow \infty} |\partial_N q_{n_l}(((9/40)\xi, y_{+}((9/40)\xi))-wN)|^2 dw =0$$
and possibly up to another subsequence

$$
\lim_{l\rightarrow \infty} \partial_N q_{n_l}(((9/40)\xi, y_{+}((9/40)\xi))-wN)
$$
 is continuous at $w=0$ thanks to the $W^{2,p}$ estimates: for $p$ large, $W^{2,p}$ is compactly embedded in $C^{1,\alpha}$; next, since

$$
\lim_{l\rightarrow \infty} |\partial_N q_{n_l}(((9/40)\xi, y_{+}((9/40)\xi))-wN)|=0
$$
a.e. in $[0,|a|]$, 
\begin{align*}
\lim_{l \rightarrow \infty}& |\partial_N q_{n_l}((9/40)\xi, y_{+}((9/40)\xi))|\\
&=\lim_{w\rightarrow 0} \lim_{l\rightarrow \infty} |\partial_N q_{n_l}(((9/40)\xi, y_{+}((9/40)\xi))-wN)=0, 
\end{align*}
a contradiction with \eqref{h}. Therefore,

\begin{align*}
 \liminf_{n \rightarrow \infty} |\partial_N q_n((9/40)\xi, y_{+}((9/40)\xi))| &=\limsup_{n \rightarrow \infty} |\partial_N q_n((9/40)\xi, y_{+}((9/40)\xi))|\\
 &= \lim_{n \rightarrow \infty} \partial_N q_n((9/40)\xi, y_{+}((9/40)\xi))=0.
\end{align*}
In particular, \eqref{z1} and $q_n=\phi_n-w_n$ imply

$$
\partial_{N} \phi_n((9/40)\xi, y_{+}((9/40)\xi)) \rightarrow 0
$$
as $n \rightarrow \infty$; hence, this yields

$$
\Big|\partial_{N} \Big(u_n((9/40)\xi, y_{+}((9/40)\xi))-\frac{-p}{\sqrt{\pi}}\Big)\Big| \rightarrow 0
$$
$\&$ there exists $\hat N \in \mathbb{N}$ such that for all $n \ge \hat N$, 

\begin{equation} \label{par}
\partial_{-N} u_n((9/40)\xi, y_{+}((9/40)\xi)) \ge \frac{1}{6}\frac{p}{\sqrt{\pi}};
\end{equation}
also,
$W^{2,p}$ \& BMO estimates imply when $1<p<\infty$, 

$$h_{\xi} \in \{(\xi_j, y_{+}(\xi_j)): \xi_j \in [\frac{\xi}{5}, \frac{\xi}{4}]\},$$ 

$$
||D^2 u_n||_{L^p(B_{a_*r}(h_\xi))} \le C_{2,p}\big(||\lambda_n u_n||_{L^p(B_{ar}(h_\xi))}+||u_n||_{L^1(B_{ar}(h_\xi))}\big),
$$

\begin{align*}
\sup_{x \in B_{a_*\xi}(h_\xi), r>0} &\frac{1}{r^2} \int_{B_r(x) \cap B_{a_*\xi}(h_\xi)} |D^2 u_n(y)-(D^2 u_n)_{r,x}|^2 dy+||D^2 u_n||_{L^2(B_{a_*\xi}(h_\xi))}^2 \\
&\le C_{\infty, 2} \big( ||\lambda_n u_n||_{L^\infty(B_{a\xi}(h_\xi))}+||u_n||_{L^1(B_{a_*\xi}(h_\xi))} \big)
\end{align*}
where $a_*<a\le 1/8$, therefore 

\begin{align*}
&|D^2 u_n(h_\xi)|^2 \le t(a_*,a)\limsup_{r \rightarrow 0} \Big [\frac{1}{r^2} \int_{B_r(h_\xi)} |D^2 u_n(y)-(D^2 u_n)_{r,h_\xi}|^2 dy \\
&+\frac{1}{r^2} \Big(C_{2,2}\big(||\lambda_n u_n||_{L^2(B_{ar}(h_\xi))}+||u_n||_{L^1(B_{ar}(h_\xi))}\big)\Big)^2\Big]\\
&\le t(a_*,a)\Big[C_{\infty, 2} \big( ||\lambda_n u_n||_{L^\infty(B_{a\xi}(h_\xi))}+||u_n||_{L^1(B_{a_*\xi}(h_\xi))} \big)\\
&+\limsup_{r \rightarrow 0}\frac{1}{r^2} \Big(C_{2,2}\big(||\lambda_n u_n||_{L^2(B_{ar}(h_\xi))}+||u_n||_{L^1(B_{ar}(h_\xi))}\big)\Big)^2\Big]\\&\le A(a_*,a, ||\lambda_n u_n||_{L^\infty}, C_{\infty, 2})\\
\end{align*}
thanks to Lebesgue's differentiation theorem because $h_\xi$ is a point of continuity of $D^2 u_n$ and thus a Lebesgue point (in particular, a similar argument  implies \eqref{eal} in the context of \eqref{p4m}); since the $L^\infty$ norms are bounded as $n\rightarrow \infty$, it thus follows that 

\begin{equation} \label{zqm5}
|\partial_{-N} (u_n((9/40)\xi, y_{+}((9/40)\xi))-u_n(h_\xi))| \le \sqrt{A(a_*,a, ||\lambda_n u_n||_{L^\infty}, 2)}\xi;
\end{equation}
hence as
$$|\nabla u_n|=-N \cdot \nabla u_n$$
on $\partial P\setminus{\text{vertices}},$ and since $\xi \rightarrow 0$ as $n \rightarrow \infty$, \eqref{par} and \eqref{zqm5} imply
that there exists $\alpha_2>0$ such that
\begin{align}
 \frac{d^2 \lambda}{d t^2}&=\frac{2(\frac{b}{2}-t)}{\xi(\xi^2+(t-\frac{b}{2})^2)}\int_0^\xi \alpha |\nabla u_n(\alpha, y_+)|^2 d \alpha \notag\\ 
 &+\frac{2(\frac{b}{2}+t)}{\xi(\xi^2+(t+\frac{b}{2})^2)}\int_0^\xi \alpha |\nabla u_n(\alpha, y_{-})|^2 d \alpha \notag\\ \notag
 & \ge\frac{2(\frac{b}{2}-t)}{\xi(\xi^2+(t-\frac{b}{2})^2)} \alpha_2 \int_{\frac{\xi}{5}}^{\frac{\xi}{4}} \alpha d \alpha +\frac{2(\frac{b}{2}+t)}{\xi(\xi^2+(t+\frac{b}{2})^2)}\alpha_2 \int_{\frac{\xi}{5}}^{\frac{\xi}{4}}\alpha d \alpha\\
& \ge \Big(\frac{(\frac{b}{2}-t)}{(\xi^2+(t-\frac{b}{2})^2)}+\frac{(\frac{b}{2}+t)}{(\xi^2+(t+\frac{b}{2})^2)}\Big)\alpha_2 \frac{9\xi}{400}. \label{lower6}
\end{align}
Moreover, observe

$$
\frac{d \lambda}{dt}=\int_0^\xi \frac{\alpha}{\xi}|\nabla u_{i+1,n}(\alpha, y_{-})|^2 d\alpha-\int_0^\xi \frac{\alpha}{\xi} |\nabla u_{i+1,n}(\alpha, y_+)|^2 d\alpha \rightarrow 0,
$$
$u_{i+1,n}$ is the eigenfunction corresponding to $P^{i+1}$. The eigenfunctions around the vertex of a polygon with opening $\pi / \alpha$ have the form
$$
\alpha_1r^{\alpha}\sin(\alpha \phi)+O(r^{2\alpha})+O(r^{2+\alpha}),
$$ 
$\phi \in [0,\frac{\pi}{\alpha}]$ in the polar coordinates relative to the cone of the vertex (one obtains that by expanding the eigenfunction in terms of Bessel functions); the opening of the vertex in the n-gon associated with the integral 
has angle $$\frac{\pi}{\alpha_{i+1,n}} \approx \frac{\pi}{\alpha_n}:=\frac{\pi}{(n/(n-2))};$$
hence $\nabla u_{i+1,n}$ has a H\"older modulus, nevertheless the cusp singularity appears in both integrals in the expression for $\frac{d \lambda}{dt}$ and is cancelled thanks to the symmetry: $|\nabla^+ r^\alpha \sin(\alpha \phi)|=|\nabla^-r^\alpha \sin(\alpha \phi)|$, $\nabla^{\pm}$ is the $y_{\pm}$ gradient. 
If $t_i>0$, note $\eqref{t_{i-1}}$ then yields 

$$  
\Big|\frac{\pi}{\alpha_n}-\frac{\pi}{\alpha_{i+1,n}}\Big| \le \tilde a t_{i};
$$
supposing $u$ is the eigenfunction of $P_n$, symmetry leads to 
\begin{align*}
\Bigg |\int_{a_2\xi}^{\xi} &\frac{x}{\xi}|\nabla u_{i+1,n}(x,y_+(x))|^2dx-\int_{a_2 \xi}^\xi \frac{x}{\xi}|\nabla u_{i+1,n}(x,y_-(x))|^2dx\Bigg | \\
&=\Bigg |\Big(\int_{a_2\xi}^{\xi} \frac{x}{\xi}|\nabla u_{i+1,n}(x,y_+(x))|^2dx-\int_{a_2 \xi}^\xi \frac{x}{\xi}|\nabla u_{i+1,n}(x,y_-(x))|^2dx\Big) \\
&- \Big(\int_{a_2 \xi}^\xi \frac{x}{\xi}|\nabla u(x,y_{+,a}(x))|^2dx-\int_{a_2 \xi}^\xi \frac{x}{\xi}|\nabla u(x,y_{-,a}(x))|^2dx \Big) \Bigg |\\
& \le \tilde a_l t_{i}, \\
\end{align*}
where $y_{\pm,a}$ is associated with $P_n$ (note that since the cones are $\tilde a t_{i}$ away, the same $a_2$ can be utilized and the same coordinate system around the vertex). Therefore, there exist $a_1\approx 0,$ $a_2 \approx 1$, $a_i \in (0,1)$ s.t. 

\begin{align}
\Bigg |\int_{a_2\xi}^{\xi} &\frac{x}{\xi}|\nabla u_{i+1,n}(x,y_+(x))|^2dx-\int_{a_2 \xi}^\xi \frac{x}{\xi}|\nabla u_{i+1,n}(x,y_-(x))|^2dx \notag\\ \notag
&+\int_{0}^{a_1\xi} \frac{x}{\xi}|\nabla u_{i+1,n}(x,y_+(x))|^2dx-\int_{0}^{a_1\xi} \frac{x}{\xi}|\nabla u_{i+1,n}(x,y_-(x))|^2dx\Bigg |\\ \label{ij1}
& \le \tilde \alpha_l t_i\Big(\int_{0}^{a_1\xi}\frac{x}{\xi}dx+\int_{a_2 \xi}^{\xi}\frac{x}{\xi}dx\Big). 
\end{align}
\begin{figure}[htbp] 
\centering
\includegraphics[width=.67 \textwidth]{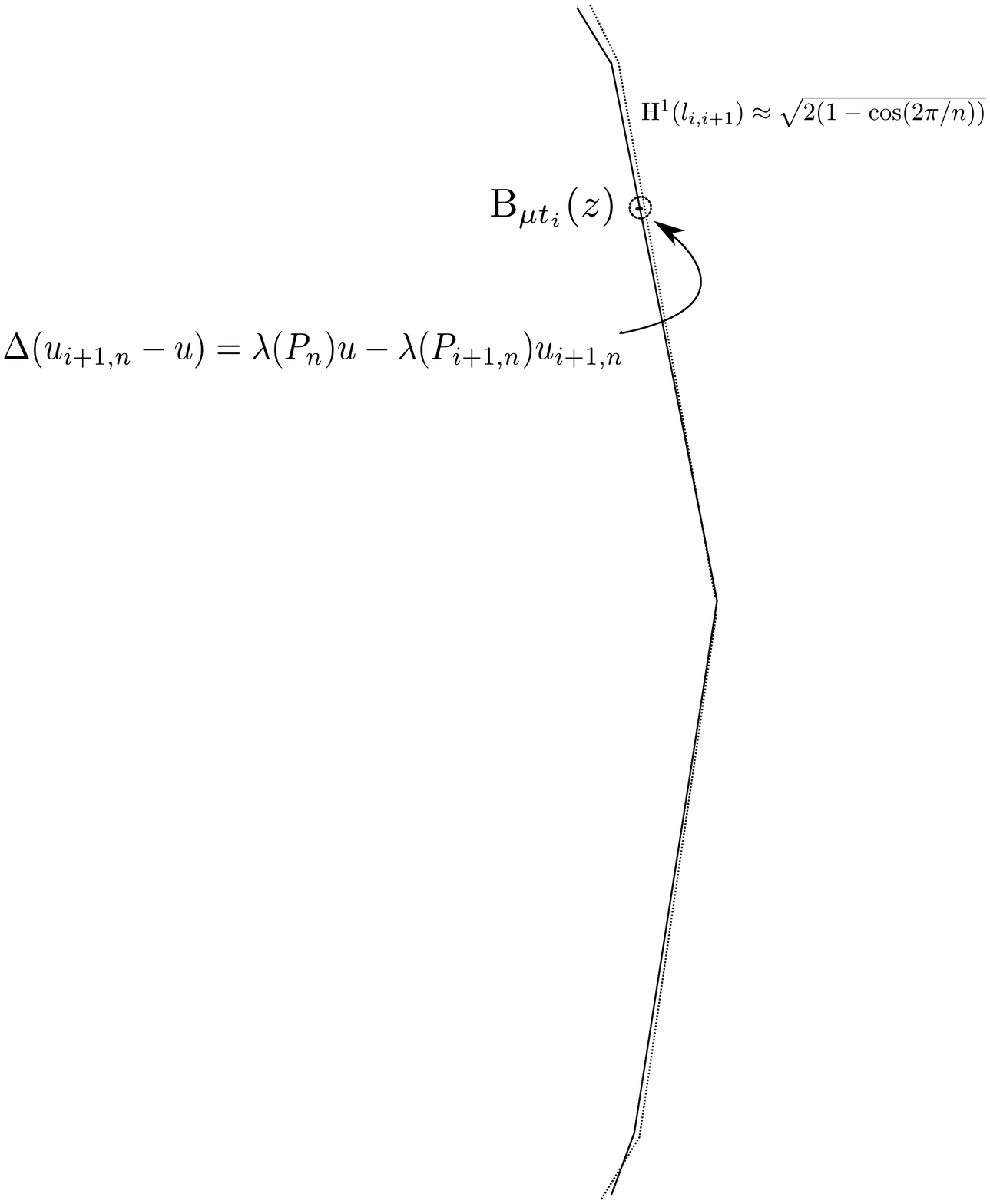}
\caption{}
\label{ot2}
\end{figure}
Note that away from the vertex, the eigenfunctions are $C^\infty$ up to the boundary and therefore can be reflected and extended smoothly in a neighborhood. If $l_{i,i+1}$ is the segment corresponding to $P^{i+1}=P_{i+1,n}$, $z \in l_{i,i+1}$, $u_{i+1,n}$ is a solution of 
$$
\Delta u_{i+1,n}=-\lambda(P_{i+1,n})u_{i+1,n}
$$ 
on $B_{\mu t_i}(z)$ where $\mu>0$ is small  \& $u_{i+1,n}$ vanishes on $l_{i,i+1}$ (Figure \ref{ot2}); this yields

$$
\Delta (u_{i+1,n}-u)=\lambda(P_n)u-\lambda(P_{i+1,n})u_{i+1,n}
$$ 
on $B_{\mu t_i}(z)$ because 
$$
\Big|\frac{\pi}{\alpha_n}-\frac{\pi}{\alpha_{i+1,n}}\Big| \le \tilde a t_i,
$$ 

$$\mathcal{H}^1(l_{i,i+1}) \approx \sqrt{2(1-\cos(2\pi/n))},$$
$t_1$ is small,  \& $t_i$ has a geometric upper bound (cf. $A_n$); \& if $x, x_a \in  B_{\frac{\mu}{2}\mu t_i}(z)$
\begin{align}
|\nabla u_{i+1,n}(x_a)-\nabla u(x)| &\le |\nabla u_{i+1,n}(x_a)-\nabla u_{i+1,n}(x)|+|\nabla u_{i+1,n}(x)-\nabla u(x)| \label{zyq}\\\notag
&\le at_i+ |\nabla u_{i+1,n}(x)-\nabla u(x)|;
\end{align}
the $C^{2,\alpha}$ boundary estimates  and the $t_i$ closeness of $P_n$ and $P_{i+1,n}$ assuming $t_i>0$ imply 

\begin{align}
 ||\nabla(u_{i+1,n}-u)||_{L^\infty(B_{\frac{\mu}{2} t_i}(z))} &\le  q\Big(||u-u_{i+1,n}||_{L^\infty(B_{\mu t_i}(z))}\notag \\ \label{ine9s}
& \hskip .8in+ ||(\lambda(P_n)u-\lambda(P_{i+1,n})u_{i+1,n})||_{L^\infty(B_{\mu t_i}(z))}\Big);
\end{align}
assume that $l_n$ denotes the edge of $P_n$, $z_n \in l_n \cap  B_{\mu t_i}(z)$,   $x \in  B_{\mu t_i}(z)$; observe the boundary data implies $u(z_n)=u_{i+1,n}(z)=0$, 
\begin{equation} \label{l3v}
|(u-u_{i+1,n})(x)|=|u(x)-u(z_n)+(u_{i+1,n}(z)-u_{i+1,n}(x))| \le q_at_i.
\end{equation}
In addition, 

\begin{align} \notag
|(\lambda(P_n)u&-\lambda(P_{i+1,n})u_{i+1,n})(x)| \label{m2e}\\ \notag
&\le |\lambda(P_n)u(x)-\lambda(P_n)u(z_n)|\notag\\
&+|\lambda(P_{i+1,n})u_{i+1,n}(z)-\lambda(P_{i+1,n})u_{i+1,n}(x)| \notag\\
&\le at_i.
\end{align}
Hence, \eqref{zyq}, \eqref{ine9s}, \eqref{l3v}, and \eqref{m2e} yield

\begin{align}
& \Big|\int_{a_1\xi}^{a_2\xi} \frac{x}{\xi}|\nabla u_{i+1,n}(x,y_+(x))|^2dx-\int_{a_1 \xi}^{a_2\xi} \frac{x}{\xi}|\nabla u_{i+1,n}(x,y_-(x))|^2dx\Big|\label{e4}\\ \notag
& \le \alpha_w \int_{a_1\xi}^{a_2\xi} \frac{x}{\xi} \Big ||\nabla u_{i+1,n}(x,y_+(x))|-|\nabla u_{i+1,n}(x,y_-(x))| \Big| dx\\ \notag
& \le \alpha_w \Big(\int_{a_1\xi}^{a_2\xi} \frac{x}{\xi} \Big ||\nabla u_{i+1,n}(x,y_+(x))|-|\nabla u(x,y_{+,a}(x))| \Big|dx+\\ \notag
& \int_{a_1\xi}^{a_2\xi}\frac{x}{\xi}\Big ||\nabla u_{i+1,n}(x,y_-(x))|-|\nabla u(x,y_{-,a}(x))| \Big |dx\Big)\\ \notag
& \le \alpha_* a_* t_i  \int_{a_1\xi}^{a_2\xi} \frac{x}{\xi}dx\le \alpha_* \bar a_1 t_i \xi 
\end{align}
(by symmetry, $|\nabla u(x,y_{+,a}(x))|=|\nabla u(x,y_{-,a}(x))|$). Therefore, it follows that with $n$ large \eqref{ij1} and \eqref{e4} imply

$$
\Big|\frac{d \lambda(P^{k})}{dt}\Big| \le   \alpha \xi t_{k-1}
$$
where $\alpha>0$ only depends on the $W^{2,p}$ bounds and  universal constants and thus supposing $t_{k-1}>0$

\begin{align*}
 \frac{d \lambda(P^{k})}{dt} &\Big |_{t=0}t_{k-1}+ \frac{d^2 \lambda(P^k)}{dt^2}\Big|_{t=0}\frac{t_{k-1}^2}{7} \\
 &\ge \Big(\frac{1}{7}\Big(\frac{(\frac{b}{2}-t)}{(\xi^2+(t-\frac{b}{2})^2)}+\frac{(\frac{b}{2}+t)}{(\xi^2+(t+\frac{b}{2})^2)}\Big)\alpha_2 \frac{9}{400}-\alpha \Big) \xi t_{k-1}^2 \\
 &> 0
\end{align*}
via \eqref{lower6}, 

$$
\xi(n)=1-\cos(2\pi/n) \approx \xi_k,
$$

$$
 b(n)=2\sin(2\pi/n) \approx b_k,
$$

$$
\xi(n)^2+\Big(\frac{b(n)}{2}\Big)^2 = 2(1-\cos(2\pi/n))=2\xi(n),
$$

$$
\frac{\frac{b}{2}}{(\xi^2+(\frac{b}{2})^2)} \approx \frac{\sin(2\pi/n)}{2(1-\cos(2\pi/n))} \rightarrow \infty.
$$
Therefore

\begin{align*}
\lambda(P)&\ge \lambda(P^\infty)+\sum_{\{k \ge 2: t_{k-1} \neq 0 \}} \frac{d \lambda(P^{k})}{dt}\Big|_{t=0}t_{k-1}+\sum_{\{k \ge 2: t_{k-1} \neq 0 \}}  \frac{d^2 \lambda(P^k)}{dt^2}\Big|_{t=0}\frac{t_{k-1}^2}{7}\\
&\ge \lambda(P_\infty)=\lambda(P_n).
\end{align*} 
Thus the above implies

$$A_n \subset$$

$$ \Big\{\text{vertices}(P) \in \bigotimes_{i=1}^{n} B_{a_n}(w_i): \sum_{k=2}^\infty  \frac{d \lambda(P^{k})}{dt}\Big|_{t=0}t_{k-1}+\sum_{k=2}^\infty  \frac{d^2 \lambda(P^k)}{dt^2}\Big|_{t=0}\frac{t_{k-1}^2}{7} \ge 0 \Big\}.$$
Define

$$
M_n(T)=\Big\{\text{vertices}(P) \in \bigotimes_{i=1}^{n} B_{a_n}(w_i): \text{$t_{k}=0$ for all $k \ge T$, }|P|=|P_n|\Big\}
$$
when $T \in \mathbb{N}$
and observe that if $P \in M_n(T)$, $P^k = P^\infty$ for $k \ge T$.
In particular, if $T>1$, $t_a=\max_{j \in \{1,\ldots, T-1\}} t_j$,

\begin{align*}
\alpha(n) |R(P^1) \Delta P_n|^2 \le \delta(P^1)&=L^2(P^1)-L^2(P_n)\\
&\le \bar a \xi^{\frac{1}{2}}(n) \Big(L(P^i)+L(P_n)\Big)  \sum_{j = 1}^{T-1} t_j^2 \\
&\le (T-1) \omega t_a^2.
\end{align*}
The inequality implies 
 
 $$
\Big|\frac{d \lambda(P^{k})}{dt}\Big|_{t=0}t_{k-1}\Big| \le r_* t_a^2
 $$
where $r_*$ is independent of $n, k$ thanks to the $W^{2,p}$ argument and therefore
 \begin{align*}
\sum_{k \ge 2} &\frac{d \lambda(P^{k})}{dt}\Big|_{t=0}t_{k-1}+\sum_{k \ge 2}  \frac{d^2 \lambda(P^k)}{dt^2}\Big|_{t=0}\frac{t_{k-1}^2}{7}\\
&\ge \sum_{\{k-1 \neq a\}}  \frac{d^2 \lambda(P^k)}{dt^2}\Big|_{t=0}\frac{t_{k-1}^2}{7}  +\Big( \frac{d^2 \lambda(P^{a+1})}{dt^2}\Big|_{t=0}\frac{1}{7}-r_*(T-1)\Big)t_{a}^2.
\end{align*}
Observe that there exists $N \in \mathbb{N}$ sufficiently large such that if $n \ge N$,

$$
\frac{d^2 \lambda(P^{a+1})}{dt^2}\Big|_{t=0}\frac{1}{7}-r_*(T-1) \ge 0
$$
via \eqref{lower6} as a result of

$$
\frac{\frac{b}{2}}{(\xi^2+(\frac{b}{2})^2)} \approx \frac{\sin(2\pi/n)}{2(1-\cos(2\pi/n))} \rightarrow \infty.
$$
In particular, one may let $T=T(n)=An$ where $A>0$ depends on $r_*$ and universal constants. If 

\begin{align*}
&A_{Eq(n)}=\Big\{\text{vertices}(P) \in \bigotimes_{i=1}^{n} B_{a_n}(w_i): \sum_{j \ge k} t_j^2\le \alpha t_{k-1}^2\\
& \text{ whenever $t_{k-1}>0$}, P^\infty \neq P^i \text{ if $i \in \mathbb{N}$}\Big\}
\end{align*}
thanks to $P^\infty$ being always equilateral, this yields 

\begin{equation} \label{ea} 
\inf_{P \in M_n(T) \cup A_{Eq(n)}} \lambda(P) 
= \inf_{\Big\{\text{vertices}(P) \in \bigotimes_{i=1}^{n} B_{a_n}(w_i): \text{ P is equilateral, } |P|=|P_n|\Big\}} \lambda(P). 
\end{equation}
The above then implies

$$
\Big(M_n(T) \cap \{P: P^\infty = P_n\} \Big) \cup A_n \subset
$$
\begin{align*}
\Big\{\text{vertices}(P) \in \bigotimes_{i=1}^{n} B_{a_n}(w_i):& \sum_{k=2}^\infty  \frac{d \lambda(P^{k})}{dt}\Big|_{t=0}t_{k-1}\\
&+\sum_{k=2}^\infty  \frac{d^2 \lambda(P^k)}{dt^2}\Big|_{t=0}\frac{t_{k-1}^2}{7} \ge 0, \text{ $P^\infty = P_n$} \Big\}.
\end{align*}
Therefore
\begin{align*}
&\inf_{\Big\{P \in \big(M_n(T) \cap \{P: P^\infty = P_n\} \big) \cup A_n, |P|=|P_n|\Big\}} \lambda(P) \\
&=\inf_{\Big\{\text{vertices}(P) \in \bigotimes_{i=1}^{n} B_{a_n}(w_i): \sum_{k=2}^\infty  \frac{d \lambda(P^{k})}{dt}|_{t=0}t_{k-1}+\sum_{k=2}^\infty  \frac{d^2 \lambda(P^k)}{dt^2}|_{t=0}\frac{t_{k-1}^2}{7} \ge 0, \text{ $P^\infty = P_n$}\Big \}}  \lambda(P)\\
&= \lambda(P_n).
\end{align*}
Note that  the smallness of $a_n$ \& the convexity and non-degeneracy of $P_n$ imply convexity of all $n$-gons in a neighborhood of $P_n$ relative to $\mathcal{M}(n, |P_n|)$:
assume not, then let $v_1$, $v_2$, $v_3$ be the vertices which generate a non-convex subset. Let the segment $\overrightarrow{v_1v_2}$ generate the line such that $v_3$ is on one side of it 
$$\{z: (z-v_1) \cdot \overrightarrow{v_1v_2}^{\perp} > 0\}$$     
and there is another vertex $v$ on the other side. Note that there exists $\mu_n$ small depending only on $P_n$ such that if one considers neighborhoods around its vertices and $v_i$ are in the neighborhoods with radii $\mu_n$, the line obtained via  $\overrightarrow{v_1v_2}$ will have $v_3$ on 
$$\{z: (z-v_1) \cdot \overrightarrow{v_1v_2}^{\perp} < 0\}$$ which contradicts the above.

\end{proof}

\begin{rem}
The proof encodes $N$ in an explicit manner: it is complicated, albeit depends on computable constants.
\end{rem}

\subsection{Proof of Corollary \ref{pa}}

If $\{E_k\}\subset \mathcal{M}(k, \pi)$,
$$
\infty>\limsup_{k\rightarrow \infty} \inf_{x \in \mathbb{R}^2} |E_k\Delta (B_1+x)|>0,
$$
it then follows that
$$
\limsup_{k\rightarrow \infty} \lambda(E_k)>\lambda(B_1).
$$
If not, then 
$$
\limsup_{k\rightarrow \infty} \lambda(E_k)=\lambda(B_1)
$$
via the Faber-Krahn theorem \cite{MR1512244, MR2w264470op} which implies that 
\begin{equation} \label{uz@}
\lim_{k\rightarrow \infty} \lambda(E_k)=\lambda(B_1).
\end{equation}
Suppose 
\begin{equation} \label{p1m}
\lim_{l\rightarrow \infty} \inf_{x \in \mathbb{R}^2} |E_{k_l}\Delta (B_1+x)|=\limsup_{k\rightarrow \infty} \inf_{x \in \mathbb{R}^2} |E_k\Delta (B_1+x)|>0.
\end{equation}
Hence one obtains mod translations 

$$
E_{k_l} \subset B_R
$$
for an $R>0$: assuming $E_{k_l}$ a convex $k_l$-gon which contains a side approaching $\infty$, since $|E_{k_l}|=\pi$, if $l$ is large, 

$$
E_{k_l} \subset R_{k_l}
$$
where $R_{k_l}$ is a rectangle with one side $s_{k_l}$ approaching $0$. The eigenvalue of $R_{k_l}$ is:

$$
\pi^2\Big(\frac{1}{S_{k_l}^2}+\frac{1}{s_{k_l}^2}\Big) \rightarrow \infty,
$$
therefore this contradicts

$$
\lambda(R_{k_l}) < \lambda(E_{k_l}).
$$
One may, with a similar argument, remove the convexity assumption to generate a larger class. The uniform boundedness of $\{E_{k_l}\}$ implies a converging subsequence (still expressed as $E_{k_l}$)

$$E_{k_l} \rightarrow E$$ 

\noindent with $|E|=\pi$. In particular via Chenais' theorem \cite[Theorem 2.3.18]{MR2251558},

$$\lambda(E_{k_l}) \rightarrow \lambda(E)$$
\& via the Faber-Krahn theorem thanks to \eqref{p1m},

$$
 \lambda(E)>\lambda(B_1),
$$
contradicting \eqref{uz@}.
Therefore, 

$$
\lim_{k\rightarrow \infty} \lambda(E_k) = \lambda(B_1)
$$
then implies that 
$$
\lim_{k\rightarrow \infty} \inf_{x \in \mathbb{R}^2}|E_k \Delta (B_1+x)|=0
$$
and there exists $\omega(0^+)=0$ so that assuming $|E|=\pi$,

$$
\lambda(E)-\lambda(B_1) \ge \omega( \inf_{x \in \mathbb{R}^2}|E \Delta (B_1+x)|).
$$
Let

$$e_k=\lambda(P_k)-\lambda(B_1)>0$$ 

$$
\Omega_k=\{E \in \mathcal{M}(k, \pi): \omega(|E \Delta (B_1+x)|) \ge e_k \text{     for some $x \in \mathbb{R}^2$}\}.
$$
Therefore supposing $E \in \Omega_n$,

$$
\lambda(E)-\lambda(P_n) \ge \omega(|E \Delta (B_1+x_E)|) - e_n \ge 0.
$$

\subsection{Proof of: Theorem \ref{E}, Corollary \ref{Y}, Corollary \ref{E_3}, and Corollary \ref{E_4}}
Observe that the proof of Theorem \ref{a} and an application of the mean-value theorem (via Taylor's theorem) implies 
$$
 \lambda(P)=\lambda(P_{\text{eq}})+\sum_{k=2}^\infty  \frac{d \lambda(P^{k})}{dt}\Big|_{t=0}t_{k-1}+\sum_{k=2}^\infty  \frac{d^2 \lambda(P^k)}{dt^2}\Big|_{t=t_{e_{k-1}}}\frac{t_{k-1}^2}{2},
$$
where $P^1=P$, $P^{k}=(P^{k-1})^{*}$, 

$$
\frac{d \lambda}{dt}\Big|_{t=0}=\int_0^{\xi_k} \frac{\alpha}{\xi_k}|\nabla u_{n,k}(\alpha, y_{-})|^2 d\alpha-\int_0^{\xi_k} \frac{\alpha}{\xi_k} |\nabla u_{n,k}(\alpha, y_+)|^2 d\alpha,
$$

\begin{align*}
\frac{d^2 \lambda}{d t^2}\Big|_{t=t_{e_{k-1}}}&=\frac{2(\frac{b_k}{2}-t_{e_{k-1}})}{\xi_k(\xi_k^2+(t_{e_{k-1}}-\frac{b_k}{2})^2)}\int_0^{\xi_k} \alpha |\nabla u_{n,t_{e_{k-1}}}(\alpha, y_+)|^2 d \alpha\\
&+\frac{2(\frac{b_k}{2}+t_{e_{k-1}})}{\xi_k(\xi_k^2+(t_{e_{k-1}}+\frac{b_k}{2})^2)}\int_0^{\xi_k} \alpha |\nabla u_{n,t_{e_{k-1}}}(\alpha, y_{-})|^2 d \alpha,
\end{align*}
$t_{e_{k-1}} \in [0, t_{k-1}]$. 
Hence it suffices to prove $\mathcal{E}(n, \alpha)$ is an $(n-3)$--dimensional submanifold of $\mathcal{M}(n, \alpha)$. Observe that $\mathcal{M}(n, \alpha)$ is an $(2n-4)$--dimensional manifold; set
\begin{align*}
&Q(x_1,x_2,\ldots, x_n, r_1, r_2, \ldots, r_n)\\
&=\Bigg( \sum \limits _{i=1}^{n} x_i-2 \pi \hskip .08 in,\hskip .08 in \sum \limits_{i=1}^n r_i r_{i+1} \sin x_i- \sum \limits_{i=1}^n \sin \frac{2\pi}{n}\hskip .08 in,\hskip .08 in \sum \limits _{i=1}^{n} r_i \cos \left(\sum \limits_{k=1}^{i-1} x_k \right)\hskip .08 in,\\
&\hskip .8 in  \sum \limits _{i=1}^{n} r_i \sin\left(\sum \limits_{k=1}^{i-1} x_k \right)\Bigg)
\end{align*}
$$(x_*, r_*)=\left(\frac{2\pi}{n},\ldots,\frac{2\pi}{n},1,\ldots,1\right).$$

$$r_{1,n}=-\Big(\sin \frac{2\pi}{n}+\sin \frac{4\pi}{n}+\ldots+\sin \frac{2\pi(n-1)}{n}\Big)$$

$$r_{2,n}=-\Big(\sin \frac{4\pi}{n}+\ldots+\sin \frac{2\pi(n-1)}{n}\Big)$$

$$r_{3,n}=-\Big(\sin \frac{6\pi}{n}+\ldots+\sin \frac{2\pi(n-1)}{n}\Big)$$

$$t_{1,n}=\cos \frac{2\pi}{n}+\cos \frac{4\pi}{n}+\ldots+\cos \frac{2\pi(n-1)}{n}$$

$$t_{2,n}=\cos \frac{4\pi}{n}+\ldots+\cos \frac{2\pi(n-1)}{n}$$

$$t_{3,n}= \cos \frac{6\pi}{n}+\ldots+\cos \frac{2\pi(n-1)}{n}$$
\[E_{4 \times n}=
\begin{pmatrix}
  1 &  1 & 1 & \ldots  &1 &1 \\
\cos \frac{2\pi}{n} &\cos \frac{2\pi}{n}  & \cos \frac{2\pi}{n} &  \ldots & \cos \frac{2\pi}{n}& \cos \frac{2\pi}{n}\\
r_{1,n} & r_{2,n} & r_{3,n}&\ldots &-\sin \frac{2\pi(n-1)}{n} & 0 \\
 t_{1,n}& t_{2,n}&t_{3,n}&\ldots &\cos \frac{2\pi(n-1)}{n}& 0 
 \end{pmatrix} 
\]

$J_{4 \times n}$=
\begin{equation*}
\begin{pmatrix}
  0 & 0 & 0 & 0 & \ldots & 0 & 0 \\
  2\sin \frac{2\pi}{n}& 2\sin \frac{2\pi}{n} & 2\sin \frac{2\pi}{n} & 2\sin \frac{2\pi}{n} &\ldots & 2\sin \frac{2\pi}{n} & 2\sin \frac{2\pi}{n} \\
1 & \cos \frac{2\pi}{n}   &\cos \frac{4\pi}{n} &\cos \frac{6\pi}{n}&     \ldots & \cos \frac{2\pi(n-2)}{n}& \cos \frac{2\pi(n-1)}{n}\\
1 & \sin \frac{2\pi}{n}   &\sin \frac{4\pi}{n} &\sin \frac{6\pi}{n}&     \ldots & \sin \frac{2\pi(n-2)}{n}& \sin \frac{2\pi(n-1)}{n} \\
 \end{pmatrix}.
 \end{equation*}
A calculation implies

\[ D Q (x_*, r_*)=
\begin{pmatrix}
  E & J  
 \end{pmatrix} _{4 \times 2n}.
\]
Assume $(x; r) \in \text{ker}(D Q)$. It thus follows that

$$
x_1+\dots+ x_n=0
$$

$$
r_1+\dots + r_n=0
$$

$$
r_{1,n}x_1+\dots+r_{n-1,n}x_{n-1}+r_1+r_2 \cos \frac{2\pi}{n}+\dots+r_n \cos \frac{2\pi(n-1)}{n}=0
$$

$$
t_{1,n}x_1+\dots+t_{n-1,n}x_{n-1}+r_1+r_2 \sin \frac{2\pi}{n}+\dots+r_n \sin \frac{2\pi(n-1)}{n}=0.
$$
Now, let 

$$
x_1=-\Big[x_2+\dots+x_n\Big]
$$

$$
r_1=-\Big[r_2+\dots+r_n\Big];
$$
note 

\begin{align*}
x_2  &= \frac{1}{r_{2,n}-r_{1,n}}\Big[ r_{1,n}\Big [x_3+\dots+x_n\Big]-\Big[r_{3,n}x_3+\dots+r_{n-1,n}x_{n-1}\\
&+r_1+r_2 \cos \frac{2\pi}{n}+\dots+r_n \cos \frac{2\pi(n-1)}{n}\Big]\Big]
\end{align*}
and one may compute the coefficient multiplying $x_3$ 

$$
\frac{t_{2,n}-t_{1,n}}{r_{2,n}-r_{1,n}}(r_{1,n}-r_{3,n})-t_{1,n}+t_{3,n} =\cot \frac{2\pi}{n}\Big[\sin \frac{2\pi}{n}+\sin \frac{4\pi}{n}\Big]-\cos \frac{2\pi}{n}-\cos \frac{4\pi}{n} \neq 0
$$
because 
$$
\tan\frac{2\pi}{n} \neq \tan\frac{4\pi}{n}.
$$
In particular, $x_3=x_3(x_4,\ldots, x_n, r_2,\ldots,r_n)$ 
and the kernel is $(2n-4)$--dimensional. Therefore $
Q^{-1}(0,0, \ldots, 0, 0, 0, 0)
$
is locally a $(2n-4)$--dimensional manifold.\\
Set 
\begin{align*}
&\Psi(x_1,x_2,\ldots, x_n, r_1, r_2, \ldots, r_n, s)=\\
&\Bigg( \sum \limits _{i=1}^{n} x_i-2 \pi, \hskip .08 in r_2^2+r_1^2-2r_2 r_1 \cos(x_1)-s \hskip  .08 in,\hskip .08 in \ldots \hskip .08 in, \hskip .08 in r_1^2+r_n^2-2r_1 r_n \cos(x_n)-s\hskip .08 in,  \\
& \sum \limits_{i=1}^n r_i r_{i+1} \sin x_i- \sum \limits_{i=1}^n \sin \frac{2\pi}{n}\hskip .08 in,\hskip .08 in \sum \limits _{i=1}^{n} r_i \cos \left(\sum \limits_{k=1}^{i-1} x_k \right)\hskip .08 in,\hskip .08 in  \sum \limits _{i=1}^{n} r_i \sin\left(\sum \limits_{k=1}^{i-1} x_k \right)\Bigg)
\end{align*}
$$(x_*, r_*, s_*)=\left(\frac{2\pi}{n},\ldots,\frac{2\pi}{n},1,\ldots,1, 2 \sin \frac{\pi}{n}\right)$$

$$r_{1,n}=-\Big(\sin \frac{2\pi}{n}+\sin \frac{4\pi}{n}+\ldots+\sin \frac{2\pi(n-1)}{n}\Big)$$

$$r_{2,n}=-\Big(\sin \frac{4\pi}{n}+\ldots+\sin \frac{2\pi(n-1)}{n}\Big)$$

$$r_{3,n}=-\Big(\sin \frac{6\pi}{n}+\ldots+\sin \frac{2\pi(n-1)}{n}\Big)$$

$$t_{1,n}=\cos \frac{2\pi}{n}+\cos \frac{4\pi}{n}+\ldots+\cos \frac{2\pi(n-1)}{n}$$

$$t_{2,n}=\cos \frac{4\pi}{n}+\ldots+\cos \frac{2\pi(n-1)}{n}$$

$$t_{3,n}= \cos \frac{6\pi}{n}+\ldots+\cos \frac{2\pi(n-1)}{n}$$
\[N_{(n+4) \times n}=
\begin{pmatrix}
  1 &  1 & 1 & \ldots  &1 &1 \\
  2\sin \frac{2\pi}{n} & 0 & 0 &  \ldots&  0 & 0 \\
0 & 2\sin \frac{2\pi}{n} & 0 &  \ldots &0 &0\\
    \vdots  &  \ldots   &  & \ddots  &\vdots  & \vdots  \\
    0 & 0 & 0 &  \ldots & 2\sin \frac{2\pi}{n}& 0 \\
0 & 0 & 0 &  \ldots & 0 & 2\sin \frac{2\pi}{n} \\
\cos \frac{2\pi}{n} &\cos \frac{2\pi}{n}  & \cos \frac{2\pi}{n} &  \ldots & \cos \frac{2\pi}{n}& \cos \frac{2\pi}{n}\\
r_{1,n} & r_{2,n} & r_{3,n}&\ldots &-\sin \frac{2\pi(n-1)}{n} & 0 \\
 t_{1,n}& t_{2,n}&t_{3,n}&\ldots &\cos \frac{2\pi(n-1)}{n}& 0 
 \end{pmatrix} 
\]

$O_{(n+4) \times (n+1)}$=
\begin{equation*}
\begin{pmatrix}
  0 & 0 & 0 & 0 & \ldots & 0 & 0& 0 \\
  2(1-\cos \frac{2 \pi}{n}) & 2(1-\cos \frac{2 \pi}{n}) & 0 & 0 & \ldots & 0&  0 & -1 \\
0& 2(1-\cos \frac{2 \pi}{n}) & 2(1-\cos \frac{2 \pi}{n}) & 0 & \ldots &0 & 0 & -1 \\
    \vdots  &   \ldots   & \ddots  &  \ddots  & &    &   \vdots &   \vdots\\
 0& 0 & 0 &0& \ldots &2(1-\cos \frac{2 \pi}{n}) & 2(1-\cos \frac{2 \pi}{n}) & -1\\
 2(1-\cos \frac{2 \pi}{n}) & 0 & 0 &0& \ldots &0& 2(1-\cos \frac{2 \pi}{n}) & -1\\
 2\sin \frac{2\pi}{n}& 2\sin \frac{2\pi}{n} & 2\sin \frac{2\pi}{n} & 2\sin \frac{2\pi}{n} &\ldots & 2\sin \frac{2\pi}{n} & 2\sin \frac{2\pi}{n}& 0 \\
1 & \cos \frac{2\pi}{n}   &\cos \frac{4\pi}{n} &\cos \frac{6\pi}{n}&     \ldots & \cos \frac{2\pi(n-2)}{n}& \cos \frac{2\pi(n-1)}{n}& 0 \\
1 & \sin \frac{2\pi}{n}   &\sin \frac{4\pi}{n} &\sin \frac{6\pi}{n}&     \ldots & \sin \frac{2\pi(n-2)}{n}& \sin \frac{2\pi(n-1)}{n}& 0 \\
 \end{pmatrix}.
 \end{equation*}
A calculation implies 

\[ D \Psi (x_*, r_*, s_*)=
\begin{pmatrix}
  N &  O 
 \end{pmatrix} _{(n+4) \times (2n+1)}.
\]
Assume $(x; r; s) \in \text{ker}(D \Psi)$. Let  $\alpha= 2\sin \frac{2\pi}{n}$, $w=2(1-\cos \frac{2 \pi}{n})$; \\
it thus follows that

$$
x_1+\dots+ x_n=0
$$

$$
\alpha x_1+w r_1+wr_2-s=0
$$

$$
\alpha x_2+w r_2+wr_3-s=0
$$

\hskip 3 in \vdots

$$
\alpha x_n+w r_n+wr_{1}-s=0.
$$
Note that thanks to the $n+1$ equations, $x_j$ where $j \in \{1,2,\ldots,n\}$ may be written in terms of the variables $r_1, r_2,\ldots,r_n$. Therefore
since utilizing

$$
r_1+\dots + r_n=0
$$

$$
r_{1,n}x_1+\dots+r_{n-1,n}x_{n-1}+r_1+r_2 \cos \frac{2\pi}{n}+\dots+r_n \cos \frac{2\pi(n-1)}{n}=0
$$

$$
t_{1,n}x_1+\dots+t_{n-1,n}x_{n-1}+r_1+r_2 \sin\frac{2\pi}{n}+\dots+r_n \sin \frac{2\pi(n-1)}{n}=0
$$
allows elimination of another $3$ variables, the kernel may be expressed in terms of $n-3$ variables.
Hence $
\Psi^{-1}(0,0, \ldots, 0, 0, 0, 0)
$
is locally an $(n-3)$--dimensional manifold. This then yields Theorem \ref{E}. \\

\noindent Supposing $n=3$, the symmetrization implies that $T^*$ is an isosceles triangle and therefore the symmetry yields

$$
\frac{d \lambda(T^*)}{dt}\Big |_{t=0}=0.
$$
In particular, Corollary \ref{Y} then follows via iterating the process and also this implies Corollary \ref{E_3}. Observe that after one iteration, 

\begin{align*}
\frac{d^2 \lambda(T^*)}{d t^2}\Big  |_{t=0}&=\Big(\frac{2(\frac{b}{2}-t)}{\xi(\xi^2+(t-\frac{b}{2})^2)}\int_0^{\xi} \alpha |\nabla u_{3,t}(\alpha, y_+)|^2 d \alpha\\
&\hskip .9in +\frac{2(\frac{b}{2}+t)}{\xi_k(\xi^2+(t+\frac{b}{2})^2)}\int_0^{\xi} \alpha |\nabla u_{3,t}(\alpha, y_{-})|^2 d \alpha \Big)  \Big|_{t=0}\\
&=\frac{2b}{\xi(\xi^2+(\frac{b}{2})^2)}\int_0^{\xi} \alpha |\nabla u(\alpha, y_+)|^2 d \alpha
\end{align*}
via symmetry and this proves Corollary \ref{E_4} because $2\rho=b \xi$.

\subsection{Proof of Theorem \ref{Yzk}}
Since $\delta_{\lambda}(aP)=\delta_{\lambda}(P)$ and $\delta_{\mathcal{P}}(aP)=\delta_{\mathcal{P}}(P)$ if $a>0$, assume without loss of generality that $|P|=|P_3|$.  Next suppose 
$$\max\{\delta_{\mathcal{P}}(P), \delta_{\lambda}(P)\} \rightarrow \infty;$$
note that there exists a rectangle containing $P$ such that a side approaches $\infty$ and a rectangle inside $P$ such that a side approaches $\infty$ where the second side is a fixed fraction of the first, say the side with length $a$. Since the triangle $P$ has fixed area, 

$$
\lambda(P) \approx a^2;
$$
hence since 

$$
L^2(P) \approx a^2
$$
the statement in the theorem is true (observe that $\approx$ refers to an upper and lower bound only depending on universal constants). Therefore without loss

$$
\delta_{\lambda}(P)= |P|\lambda(P)-|P_3|\lambda(P_3) \le a_*,
$$ 

$$
\delta_{\mathcal{P}}(P)=\frac{L^2(P)}{12\sqrt{3}|P|}-1 \le a_*.
$$
In the remaining argument, one may remove the area constraint $|P|=|P_3|$ (for the explicit scaling exposition). Observe that if $|P|=|sP_{3}|$,

$$
\inf_{R} \frac{|R(P) \Delta sP_{3}|}{|P|} \ge w,
$$
Corollary \ref{Y}, \eqref{ozm}, and \eqref{ai@} imply

$$\delta_{\mathcal{P}}(P) \ge z_1 w^2$$
\&

$$
\delta_{\lambda}(P) \ge z_2 w^2.
$$
In particular

$$
\frac{\delta_{\mathcal{P}}(P)}{\delta_{\lambda}(P) } \le \frac{a_*}{z_2 w^2}
$$
\&

$$
\frac{\delta_{\lambda}(P)}{\delta_{\mathcal{P}}(P) } \le \frac{a_*}{z_1 w^2}.
$$
Hence one may assume some rotation exists s.t. 

$$
\frac{|R(P) \Delta sP_{3}|}{|P|} \le w
$$
where $w>0$ is small. Thus the sides of $P$ are close to 

$$
\sqrt{|sP_3|}\frac{2}{3^{1/4}};
$$
this implies $t>0$ is sufficiently small (in Corollary \ref{Y}) and $b_k \approx \sqrt{|sP_3|}\frac{2}{3^{1/4}}$, $\xi_k \approx \frac{\sqrt{3}}{2}b_k$. This then implies
\begin{align*}
\frac{2(\frac{b_k}{2}-t)}{\xi_k(\xi_k^2+(t-\frac{b_k}{2})^2)}\int_0^{\xi_k} & \alpha |\nabla u_{n,t}(\alpha, y_+)|^2 d \alpha\\
&+\frac{2(\frac{b_k}{2}+t)}{\xi_k(\xi_k^2+(t+\frac{b_k}{2})^2)}\int_0^{\xi_k} \alpha |\nabla u_{n,t}(\alpha, y_{-})|^2 d \alpha \\
&\ge 2e>0.
\end{align*}
In particular, define $\mu$ via 

$$
|\mu P|=|P_3|.
$$
Next note that $P^1=\mu P$ \& $P^\infty=P_3$, therefore via Corollary \ref{Y} and \eqref{ai@}

\begin{align*}
L^2(\mu P)-12 \sqrt{3}|\mu P|&=L^2(P^1)-L^2(P_3) \\
&=\Big(L(P^1)+L(P_3)\Big) \Big(L(P^1)-L(P_3) \Big)\\
&=\Big(L(P^1)+L(P_3)\Big) \Big(\sum_{j=1}^\infty (L(P^j)-L(P^{j +1})) \Big)\\
&\approx \sqrt{|P_3|} \sum_{j \ge 1} t_j^2 \approx \sqrt{|P_3|} (\lambda(\mu P)-\lambda(P_3))
\end{align*}
($\approx$ in the above means up to two constants which depend on $w$ and universal constants).
Hence

$$
\frac{|P_3|^2}{|P|} (L^2(P)-12 \sqrt{3}|P|) \approx \sqrt{|P_3|}(|P|\lambda(P)-|P_3|\lambda(P_3)).
$$
\begin{figure}[htbp]
\centering
\includegraphics[width=.9 \textwidth]{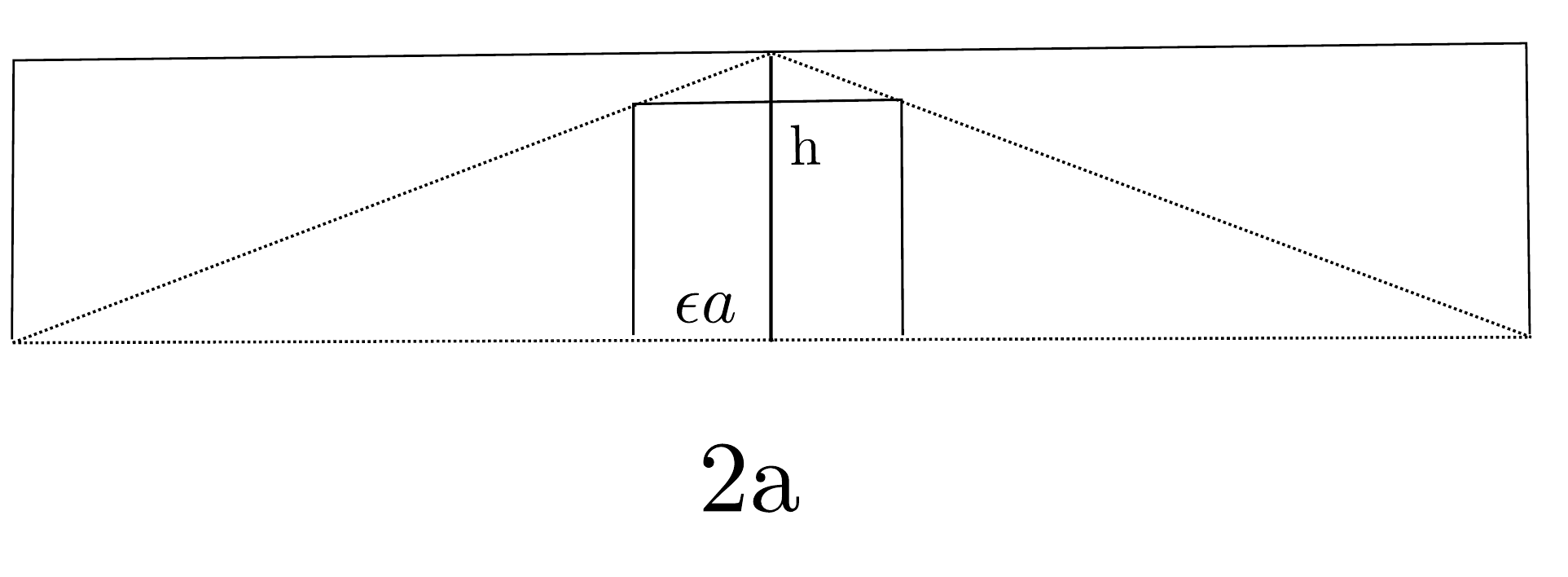}
\caption{}
 \label{wq349}
\end{figure}

\subsection{Proof of Corollary \ref{Yzk6}}
Note that the first two inequalities are immediate from Theorem \ref{Yzk}. 
Now, suppose $P$ is a triangle with sides $e$, $e$, 
 $2a \rightarrow \infty$ so that its height is $h>0$ such that $R_a$ is a rectangle inside with sides $2a\epsilon$, $(1-\epsilon)h$ also $Q_a$ is a rectangle with sides $2a$, $h$ containing $P$, Figure \ref{wq349}. Hence
 
 $$
 \pi^2\Big[\frac{1}{(1-\epsilon)^2}a^2 + \frac{1}{4(a\epsilon)^2}\Big] > \lambda(T)>         \pi^2\Big[a^2 + \frac{1}{4a^2}\Big].
 $$
Now note that $L^2(P)=(2a+2e)^2$ and 
$$
\frac{\pi^2[a^2 + \frac{1}{4a^2}]}{L^2(P)-12\sqrt{3}} \rightarrow \frac{\pi^2}{16}
$$

$$
\frac{\pi^2\Big[\frac{1}{(1-\epsilon)^2}a^2 + \frac{1}{4(a\epsilon)^2}\Big]}{L^2(P)-12\sqrt{3}} \rightarrow \frac{1}{(1-\epsilon)^2}\frac{\pi^2}{16}
$$
such that if $a$ is sufficiently large, 
 $$
 \frac{1}{(1-\epsilon)^2}\frac{\pi^2}{16}  > \frac{\lambda(T)}{L^2(P)-12\sqrt{3}} > (1-\epsilon)\frac{\pi^2}{16}.
 $$

\subsection{Proof of Theorem \ref{Yz}}
Note that since 
$$
\frac{|R(T) \Delta sT_{Eq}|}{|T|} \le 2
$$
when
$$
\delta_{\lambda}(T):=|T|\lambda(T)-|T_{Eq}|\lambda(T_{Eq}) \ge q
$$
the theorem is true and one may let $a=\frac{4}{q}$. Therefore without loss, one may consider $T$ such that $\delta_{\lambda}(T)$ is arbitrarily small. Next, suppose $T \in \mathcal{M}(3,1)$. In this case, if one of the sides is large, then one may consider a rectangle and obtain as in the proof of Theorem \ref{Yzk} a contradiction. Therefore, without loss $T \subset B_r$. Let $z>0$ and observe that supposing $B_z(T_{Eq})$ denotes a ball around $T_{Eq}$ in  $\mathcal{M}(3,1)$ (note that the assumption is $|T|=1$), one obtains thanks to \eqref{ozm} or the specific equivalence in Theorem \ref{Yzk}

$$\inf_{\{T: T \in (B_z(T_{Eq}))^c \subset  \mathcal{M}(3,1)\}} \delta_{\lambda}(T) \ge \alpha \inf_{\{T: T \in (B_z(T_{Eq}))^c \subset  \mathcal{M}(3,1)\}} \delta(T)=a_1(z)>0;$$
hence one may let $z$ be a small number so that the sides of $T$ are close to 
$$
\frac{2}{3^{1/4}};
$$
this implies $t>0$ is sufficiently small (in Corollary \ref{Y}) and $b_k \approx \frac{2}{3^{1/4}}$, $\xi_k \approx \frac{\sqrt{3}}{2}b_k$. Hence 

\begin{align*}
\frac{2(\frac{b_k}{2}-t)}{\xi_k(\xi_k^2+(t-\frac{b_k}{2})^2)}&\int_0^{\xi_k} \alpha |\nabla u_{n,t}(\alpha, y_+)|^2 d \alpha\\
&+\frac{2(\frac{b_k}{2}+t)}{\xi_k(\xi_k^2+(t+\frac{b_k}{2})^2)}\int_0^{\xi_k} \alpha |\nabla u_{n,t}(\alpha, y_{-})|^2 d \alpha \\
&\ge 2e>0.
\end{align*}
\begin{figure}[htbp] 
\centering
\includegraphics[width=.9 \textwidth]{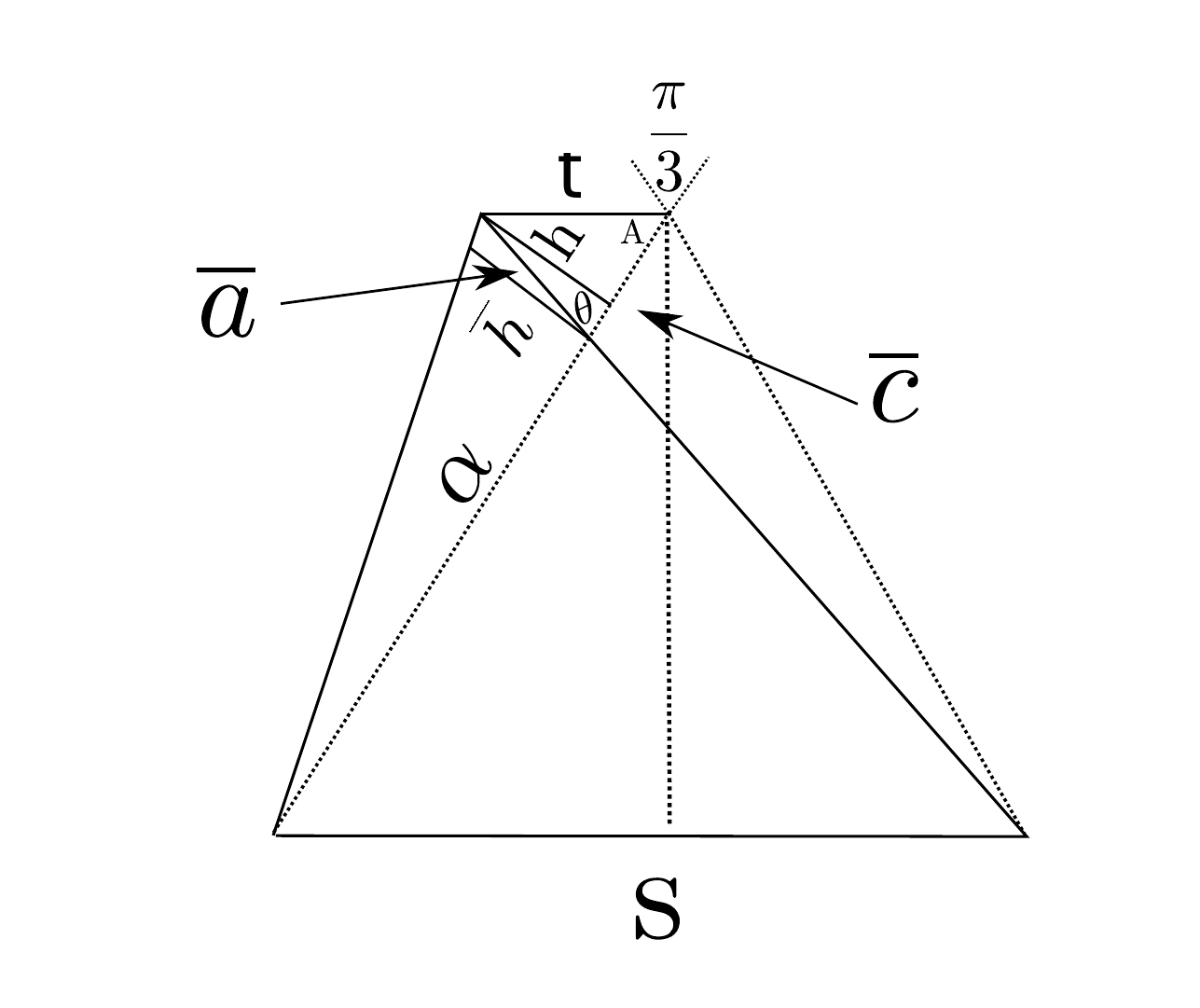}
\caption{}
\label{wq34k}
\end{figure}
In particular Corollary \ref{Y} and \eqref{ozm} therefore imply
$$
\lambda(T)-\lambda(T_{Eq}) \ge e \sum_{k=2}^\infty  t_{k-1}^2 \ge e_a \alpha(n) |R(T) \Delta T_{Eq}|^2.
$$
If $|T| \neq 1$, set $\mu=\sqrt{\frac{1}{|T|}}$ and apply the argument above to $\mu T$. Now to prove sharpness, let $P_t$ be a small perturbation of $P_3$ as in Figure \ref{wq34k}. Hence the triangle with sides $\alpha$ and $\overline{h}$ is contained in $P_t \Delta P_3$. Therefore via Corollary \ref{E_4} it is sufficient to show $\overline{h} \approx t$: observe $t \sin A =h$ and since the triangles with sides $\alpha,$ $\overline{h}$ \& $\alpha+\overline{a}\cos \theta $, $h$ are similar,
$$
\frac{\overline{h}}{\alpha}= \frac{h}{\alpha+\overline{a}\cos \theta }. 
$$ 
Therefore since if $t>0$ is small, $\theta \approx \frac{\pi}{3}$, $\alpha+\overline{c} = s$ and $\overline{c} \rightarrow 0$, the proportionality constants are obtained in terms of 

$$
 \frac{\alpha \sin A}{\alpha+\overline{a}\cos \theta }:
$$

$$\overline{h} =\alpha\frac{h}{\alpha+\overline{a}\cos \theta }=t \frac{\alpha \sin A}{\alpha+\overline{a}\cos \theta}.$$
Thus suppose there is a function $h$ with $h(a)=o(a)$ such that 

$$
|R(T) \Delta P_{Eq}|^2 \le h(\lambda(T)-\lambda(T_{Eq})); 
$$
in particular, let $T=P_t$ s.t. modulo a rotation and translation,  Corollary \ref{E_4} \& $\overline{h} \approx t$ yields

$$a_2 t^2 \le h(a_*t^2+o(t^2))$$
where $a_2, a_*>0$. 
Therefore

$$
\frac{a_2}{a_*} \le \liminf_{t \rightarrow 0^+} \frac{h(a_*t^2+o(t^2))}{a_*t^2+o(t^2)}=0, 
$$
a contradiction.

\subsection{Proof of Theorem \ref{Eu}} \label{z1k}
Theorem \ref{E} implies 

$$
 \lambda(P)=\lambda(P_{\infty})+\sum_{k=2}^\infty  \frac{d \lambda(P^{k})}{dt}\Big |_{t=0}t_{k-1}+\sum_{k=2}^\infty  \frac{d^2 \lambda(P^k)}{dt^2}\Big |_{t=t_{e_{k-1}}}\frac{t_{k-1}^2}{2}.
$$
Observe that $P_\infty$ is equilateral and hence a rhombus. Therefore Steiner symmetrization with respect to the line generated by one of the sides yields a rectangle. 
\begin{figure}[htbp] \label{wq34q}
\centering
\includegraphics[width=.9 \textwidth]{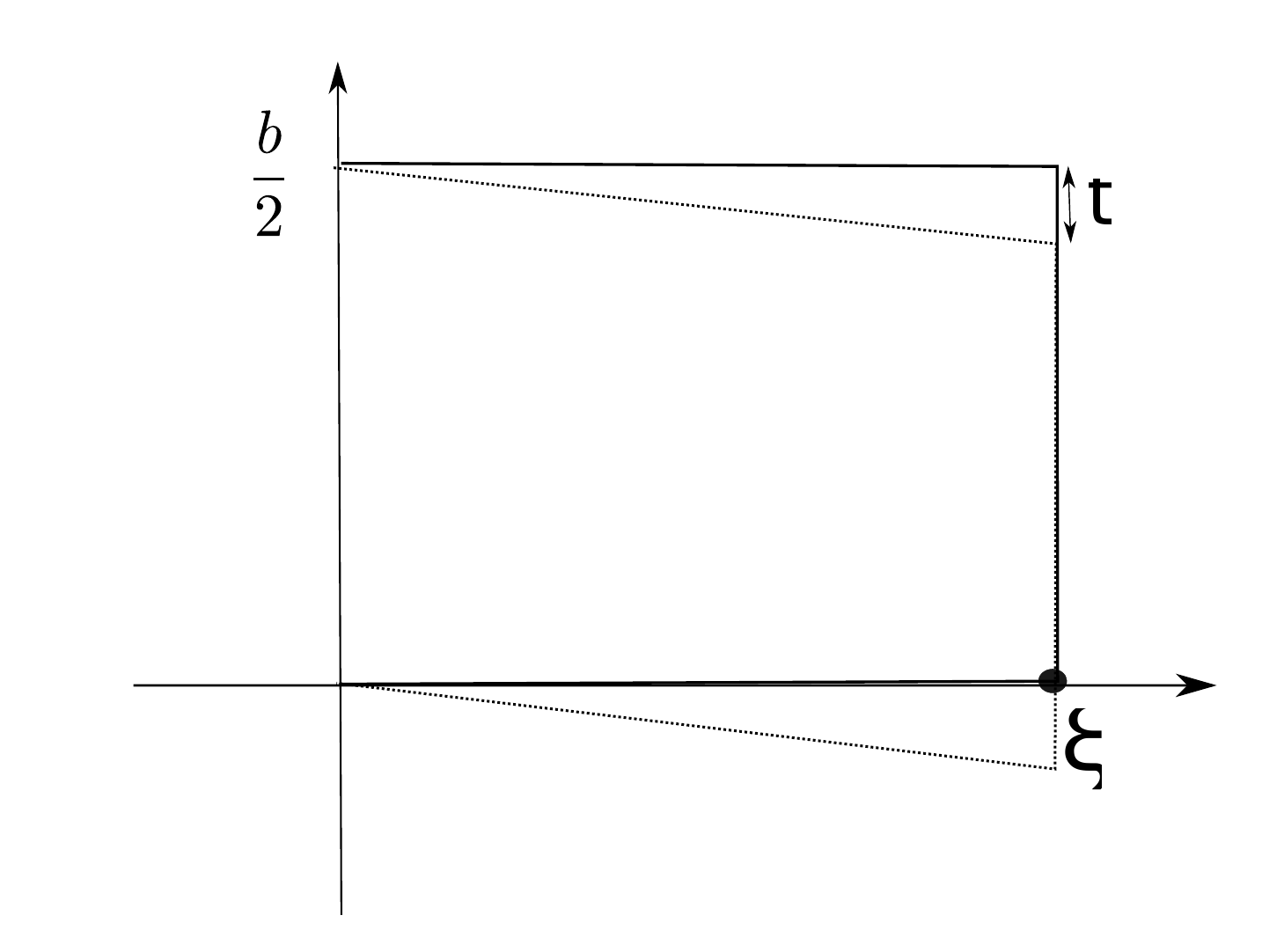}
\caption{}
\label{a12}
\end{figure}
Hence it is sufficient to calculate the eigenvalue difference between the rhombus and the rectangle:

\begin{equation} \label{v}
\frac{d \lambda}{dt}=-\int_{S_t} C(N^i \nabla_i u)^2 dS_t;
\end{equation}

 \begin{align*}
\frac{d^2 \lambda}{dt^2}&=-\int_{S_t} \frac{\delta C}{\delta t}(\nabla^i u \nabla_i u) d S_t-
&2 \int_{S_t} C(-\nabla C \cdot N)|\nabla u|^2dS_t+ \int_{S_t}C^2 B_\alpha^\alpha \nabla^i u \nabla_i u dS_t,
&\end{align*} 

{\bf The case: $y_{+}$ (the upper line-segment)} \\ 
Define
$$
z= \begin{bmatrix}
z^x(\alpha, t)
\\
z^y(\alpha,t)
\end{bmatrix}=
 \begin{bmatrix}
\alpha
\\
-\frac{t}{\xi}\alpha + \frac{b}{2}
\end{bmatrix}$$

$$
N= \begin{bmatrix}
t
\\
\xi
\end{bmatrix}\frac{1}{\sqrt{\xi^2+t^2}}
$$
$$
\partial_t \begin{bmatrix}
z^x(\alpha, t)
\\
z^y(\alpha,t)
\end{bmatrix}=
 \begin{bmatrix}
0
\\
-\frac{\alpha}{\xi}
\end{bmatrix}$$

$$v^\alpha= \frac{\alpha t}{\xi^2+t^2}$$

$$
\partial_\alpha \begin{bmatrix}
z^x(\alpha, t)
\\
z^y(\alpha,t)
\end{bmatrix}=
 \begin{bmatrix}
1
\\
-\frac{t}{\xi}
\end{bmatrix}$$

\begin{align*}
\frac{\delta z}{\delta t}&=\partial_tz-v^\alpha \partial_\alpha z\\
&=-\frac{\alpha}{\xi^2+t^2} \begin{bmatrix}
t
\\
\xi
\end{bmatrix}
\end{align*}

 \begin{equation} \label{ro}
 C=\frac{\delta z}{\delta t} \cdot N=-\frac{\alpha}{\sqrt{\xi^2+t^2}}
 \end{equation}
 
 $$ \partial_tC
=\frac{\alpha t}{(\xi^2+t^2)^{3/2}};
$$
 $$ \partial_\alpha C=-\frac{1}{\sqrt{\xi^2+t^2}}.$$
In particular, 
\begin{align*}
\frac{\delta C}{\delta t}&=\partial_tC-v^\alpha \nabla_\alpha C\\
&=2\frac{\alpha t}{(\xi^2+t^2)^{3/2}};
\end{align*}
in order to compute $\nabla C$, set $x=\alpha$, $y=-\frac{t}{\xi}x+\frac{b}{2}$:

$$
C=-\frac{x^2}{\xi\sqrt{x^2+(y-\frac{b}{2})^2}}
$$

 $$
 \nabla C=
 \begin{bmatrix}
\frac{-x^3-2x(y-\frac{b}{2})^2}{\xi(x^2+(y-\frac{b}{2})^2)^{\frac{3}{2}}}
\\
\frac{x^2(y-\frac{b}{2})}{\xi (x^2+(y-\frac{b}{2})^2)^{\frac{3}{2}}}
\end{bmatrix};
$$

\begin{align*}
\nabla C \cdot N&=  \begin{bmatrix}
\frac{-x^3-2x(y-\frac{b}{2})^2}{\xi(x^2+(y-\frac{b}{2})^2)^{\frac{3}{2}}}
\\
\frac{x^2(y-\frac{b}{2})}{\xi (x^2+(y-\frac{b}{2})^2)^{\frac{3}{2}}}
\end{bmatrix} \cdot \begin{bmatrix}
t
\\
\xi
\end{bmatrix}\frac{1}{\sqrt{\xi^2+t^2}}\\
&=-\frac{2t}{(\xi^2+t^2)};
\end{align*}
since $B_\alpha^\alpha$ has a Dirac mass at the vertex, by comparing $u$ with the solution of the corresponding equation in a disk intersecting the vertex and containing the polygon, $|\nabla u(v)|=0$. This yields
 \begin{align} \label{a1u}
&\Big(-\int_{l_{1,2}(t)} \frac{\delta C}{\delta t}(\nabla^i u \nabla_i u) dH-
2 \int_{l_{1,2}(t)} C(-\nabla C \cdot N)|\nabla u|^2dH\\\notag
&+ \int_{l_{1,2}(t)}C^2 B_\alpha^\alpha \nabla^i u \nabla_i u dH \Big)\\
&=\frac{2t}{\xi(\xi^2+t^2)}\int_0^\xi \alpha |\nabla u|^2 d \alpha.
\notag
\end{align} 

{\bf The case: $y_{-}$ (the lower line-segment)} \\ 

$$
z= \begin{bmatrix}
z^x(\alpha, t)
\\
z^y(\alpha,t)
\end{bmatrix}=
 \begin{bmatrix}
\alpha
\\
\frac{-t}{\xi}\alpha
\end{bmatrix}$$

$$
N= \begin{bmatrix}
-t
\\
-\xi
\end{bmatrix}\frac{1}{\sqrt{\xi^2+t^2}}
$$

$$
\partial_\alpha \begin{bmatrix}
z^x(\alpha, t)
\\
z^y(\alpha,t)
\end{bmatrix}=
 \begin{bmatrix}
1
\\
\frac{-t}{\xi}
\end{bmatrix}$$

$$
\partial_t \begin{bmatrix}
z^x(\alpha, t)
\\
z^y(\alpha,t)
\end{bmatrix}=
 \begin{bmatrix}
0
\\
-\frac{\alpha}{\xi}
\end{bmatrix}$$

$$v^\alpha= \frac{\alpha t}{\xi^2+t^2}$$

\begin{align*}
\frac{\delta z}{\delta t}&=\partial_tz-v^\alpha \partial_\alpha z\\
&=-\frac{\alpha}{\xi^2+t^2} \begin{bmatrix}
t
\\
\xi
\end{bmatrix}
\end{align*}

\begin{equation} \label{oz6}
C=\frac{\delta z}{\delta t} \cdot N=
\frac{\alpha}{\sqrt{\xi^2+t^2}}
\end{equation}
 
 $$ \partial_tC
=-\frac{\alpha t}{(\xi^2+t^2)^{3/2}}
$$
 $$ \partial_\alpha C=\frac{1}{\sqrt{\xi^2+t^2}}.$$
In particular, 
\begin{align*}
\frac{\delta C}{\delta t}&=\partial_tC-v^\alpha \nabla_\alpha C\\
&= -2\frac{\alpha t}{(\xi^2+t^2)^{3/2}};
\end{align*}
in order to compute $\nabla C$, set $x=\alpha$, $y=\frac{-t}{\xi}x$:

$$
C=\frac{x^2}{\xi\sqrt{x^2+y^2}}
$$

 $$
 \nabla C=
\begin{bmatrix}
\frac{x^3+2xy^2}{\xi(x^2+y^2)^{\frac{3}{2}}}
\\
-\frac{x^2y}{\xi (x^2+y^2)^{\frac{3}{2}}}
\end{bmatrix}
$$

\begin{align*}
\nabla C \cdot N&=\begin{bmatrix}
\frac{x^3+2xy^2}{\xi(x^2+y^2)^{\frac{3}{2}}}
\\
-\frac{x^2y}{\xi (x^2+y^2)^{\frac{3}{2}}}
\end{bmatrix} \cdot
\begin{bmatrix}
-t
\\
-\xi
\end{bmatrix}\frac{1}{\sqrt{\xi^2+t^2}}\\
&=\frac{-2t}{(\xi^2+t^2)}
\end{align*}
since $B_\alpha^\alpha$ has a Dirac mass at the vertex, by comparing $u$ with the solution of the corresponding equation in a disk intersecting the vertex and containing the polygon, $|\nabla u(v)|=0$. This yields
 \begin{align}
&\Big(-\int_{l_{4,3}(t)} \frac{\delta C}{\delta t}(\nabla^i u \nabla_i u) dH-
2 \int_{l_{4,3}(t)} C(-\nabla C \cdot N)|\nabla u|^2dH \label{a*5k}\\
&+ \int_{l_{4,3}(t)}C^2 B_\alpha^\alpha \nabla^i u \nabla_i u dH \Big)\notag\\
&=\frac{-2t}{\xi(\xi^2+t^2)}\int_0^\xi \alpha |\nabla u|^2 d \alpha.
\notag
\end{align}

{\bf The case: $y_{r}$ (the right line-segment)} \\ 

$$
z= \begin{bmatrix}
z^x(\alpha, t)
\\
z^y(\alpha,t)
\end{bmatrix}=
 \begin{bmatrix}
\xi
\\
\alpha-t
\end{bmatrix},$$
$\alpha \in [0,\frac{b}{2}],$
$$
N= \begin{bmatrix}
1
\\
0
\end{bmatrix}
$$

$$
\partial_\alpha \begin{bmatrix}
z^x(\alpha, t)
\\
z^y(\alpha,t)
\end{bmatrix}=
 \begin{bmatrix}
0
\\
1
\end{bmatrix}$$

$$
\partial_t \begin{bmatrix}
z^x(\alpha, t)
\\
z^y(\alpha,t)
\end{bmatrix}=
 \begin{bmatrix}
0
\\
-1
\end{bmatrix}$$

$$v^\alpha= -1$$

\begin{align*}
\frac{\delta z}{\delta t}&=\partial_tz-v^\alpha \partial_\alpha z\\
&= \begin{bmatrix}
0
\\
0
\end{bmatrix}
\end{align*}

\begin{equation} \label{oz6l}
C=\frac{\delta z}{\delta t} \cdot N=0.
\end{equation}
 Hence \eqref{ro}, \eqref{a1u}, \eqref{oz6}, \eqref{a*5k},  \& \eqref{oz6l} imply

\begin{align*}
\frac{d \lambda}{dt}&=-\int_{S_t} C(N^i \nabla_i u)^2 dS_t\\
&=\frac{1}{\xi} \Big(\int_0^\xi \alpha |\nabla u(\alpha, y_{+}(\alpha))|^2 d\alpha- \int_0^\xi \alpha |\nabla u(\alpha, y_{-}(\alpha))|^2 d\alpha\Big);
\end{align*}

$$
\frac{d^2 \lambda}{d t^2}=\frac{2t}{\xi(\xi^2+t^2)}\Big(\int_0^\xi \alpha |\nabla u(\alpha, y_{+}(\alpha))|^2 d \alpha-\int_0^\xi \alpha |\nabla u(\alpha, y_{-}(\alpha))|^2 d \alpha \Big);
$$

$$
\lambda(P_\infty)=\lambda(P_R)+\frac{d \lambda}{d t}\Big|_{t=0}t^*+\frac{d^2 \lambda}{d t^2}\Big|_{t=0}\frac{(t^*)^2}{2}+o((t^*)^2)=\lambda(P_R)+o((t^*)^2)
$$
thus one may extend $o((t^*)^2)$ to generate $\Psi$.

\subsection{Proof of Corollary \ref{o}}

There exists $N \ge 5$ such that for all $n \ge N$,   
$\alpha P_n$ minimizes $\lambda$ in $\mathcal{A}_{n}(\alpha^2 \pi)$. Moreover, the polygonal isoperimetric inequality implies that $\alpha P_n$ minimizes the perimeter in $\mathcal{M}(n, \alpha^2 \pi)$. Therefore supposing 
$$
P \in \cup_{\alpha>0} \mathcal{A}_{n}(\alpha),
$$
define $\alpha$ with $|\alpha P_n|=|P|$ so that

\begin{align*}
E(P)&=\Psi \lambda(P)+\Sigma \int_{\partial P} d\mathcal{H}^1+\Pi \int_P dA\\
& \ge \Psi \lambda(P_n) \frac{1}{\alpha^2}+\Sigma \mathcal{H}^1(\partial P_n)\alpha+\Pi |P_n| \alpha^2=h(\alpha).
\end{align*}
Thus one may optimize $h(\alpha)$ to obtain the $4^{th}$ order equation for $\alpha$.

\subsection{The rate estimate for triangles} \label{pq3}

\begin{figure}[htbp] 
\centering
\includegraphics[width=.9 \textwidth]{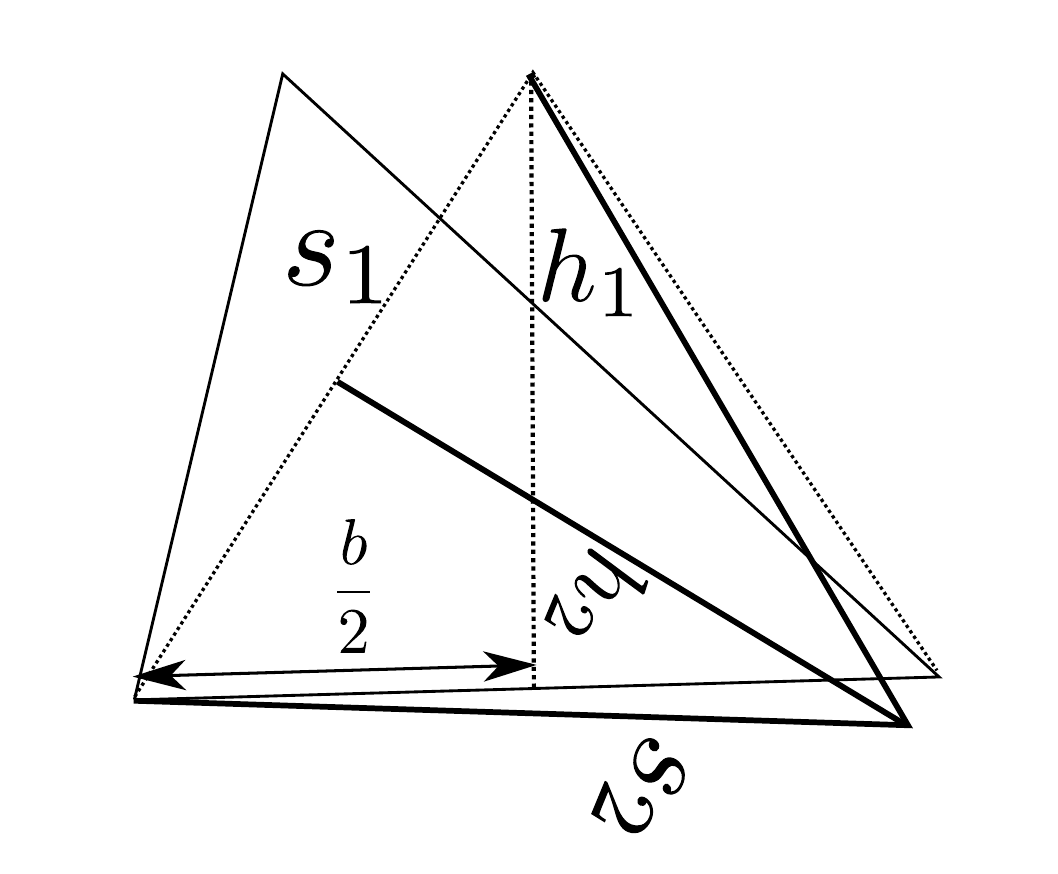}
\caption{}
\label{wqh34}
\end{figure}

\begin{thm}
Assume $A_{3}$ is the set from \eqref{o2}. Then for $a_3$ small,
$$
A_3= \Big\{\text{vertices}(P) \in \bigotimes_{i=1}^{3} B_{a_3}(w_i): P^\infty \neq P^i \text{ if $i \in \mathbb{N}$}\Big\}.
$$
\end{thm}

\begin{proof}

Observe that for any triangle $P=P^1$, $P^\infty=P_3$ (this was initially noted by P\'olya). Now,
let the first side of the triangle generated by one symmetrization be $s_1$. Observe via the area constraint (Figure \ref{wqh34}) that

\begin{align*}
s_2^2&=\Big(\frac{1}{2} \sqrt{(\frac{b}{2})^2+h_1^2}\Big)^2+\Big(\frac{bh_1}{\sqrt{(\frac{b}{2})^2+h_1^2}}\Big)^2\\
&=\frac{1}{4}s_1^2+4\frac{A^2}{s_1^2}.
\end{align*}

In particular
assume without loss that $A=1$,  

$$
\phi(s):=\frac{1}{4}s+\frac{\alpha}{s}
$$

$$
\phi'(s)=\frac{1}{4}-\frac{\alpha}{s^2},
$$
s.t. when $s_e=\frac{2}{3^{1/4}},$ $\phi(s_e^{2})=s_e^2$. Observe that if   $s \in Z_a=(s_e^2-a_0, s_e^2+a_0),$ with $a_0>0$ small,

$$
|\phi'(s)| \approx \frac{1}{2}. 
$$
Set $s_0=s \in Z_a$,
$$
s_{k+1}:=\phi(s_k)
$$
so that $s_k \rightarrow s_e^{2}$;
next considering the previous when $s_{k}^2$ is replacing $s_k$, 
define
$$
p_{\pm}(t):=\sqrt{\xi^2+(\frac{b}{2}\pm t)^2}-\sqrt{\xi^2+(\frac{b}{2})^2},
$$ 
with small $t>0$,
$$
b \approx s_e
$$

$$
\xi \approx \frac{\sqrt{3}}{2}s_e.
$$ 
By the mean value theorem
$$p_{\pm}'(t_{\pm}) t=p_{\pm}(t),$$
where $t_{\pm} \in (0, t),$ therefore since $t$ is small 
$$
|p_{\pm}'(t_{\pm})| \approx \frac{1}{2},
$$
which yields $|s_{j+1}-s_j| \approx \frac{1}{2} t_{j+1}$ ($s_1$ is the first symmetrized side and $t_1$ therefore is associated to $s_1$). In particular
$$
\frac{\sum_{j \ge k} t_j^2}{t_{k-1}^2} = \frac{t_{k}^2}{t_{k-1}^2}+\frac{t_{k+1}^2}{t_{k-1}^2}+\dots+\frac{t_{k+l}^2}{t_{k-1}^2}+  \dots
$$

\begin{align*}
\frac{t_{k}^2}{t_{k-1}^2} &\approx \frac{(2(s_k-s_{k-1}))^2}{t_{k-1}^2} \\
&= \frac{(2(s_k-s_{k-1})(s_k+s_{k-1}))^2}{(t_{k-1}(s_k+s_{k-1}))^2}\\
&=\frac{(2(\phi(s_{k-1}^2)-\phi(s_{k-2}^2)))^2}{(t_{k-1}(s_k+s_{k-1}))^2}\\
&= \frac{(2(\phi'(\mu_k) (s_{k-1}^2-s_{k-2}^2)))^2}{(t_{k-1}(s_k+s_{k-1}))^2}\\
&\approx (\phi'(\mu_k))^2 (\frac{s_{k-1}+s_{k-2}}{s_k+s_{k-1}})^2\\
& \le \alpha (\frac{1}{2}+a)^2
\end{align*}
where $a>0$ is sufficiently small,
$$\mu_k \in (\min \{s_{k-1}^2, s_{k-2}^2\}, \max \{s_{k-1}^2, s_{k-2}^2\}),$$
$$
|s_{k-1}-s_{k-2}| \approx \frac{1}{2} t_{k-1};
$$
moreover, 
\begin{align*}
\frac{t_{k+l}^2}{t_{k-1}^2} &\approx \frac{(2(\phi(s_{k+l-1}^2)-\phi(s_{k+l-2}^2)))^2}{(t_{k-1}(s_{k+l}+s_{k+l-1}))^2}\\
& \approx \Pi_{m=1}^{l+1} (\phi'(\mu_{k, m}))^2 (\frac{s_{k-1}+s_{k-2}}{s_{k+l-1}+s_{k+l}})^{2} \le a_*((\frac{1}{2}+a)^2)^{l+1}
\end{align*}
$$\mu_{k,m} \in (\min\{s_{k+m-3}^2, s_{k+m-2}^2\}, \max\{s_{k+m-3}^2, s_{k+m-2}^2\}),$$
therefore
$$
\frac{\sum_{j \ge k} t_j^2}{t_{k-1}^2} \le \alpha_3
$$
whenever $t_{k-1}>0$ \& $P^\infty \neq P_i$ if $i \in \mathbb{N}$ with a universal $\alpha_3>0$ (in the case of $\mathcal{M}(3, 1)$, observe when $t_k=t_{k+1}=0$, $P^k=P_3$, therefore the convergence is in finitely many iterations). 
\end{proof}

\subsection{Explicit examples} \label{ra7}

\begin{figure}[htbp] \label{untitled6}
\centering
\includegraphics[width=.8 \textwidth]{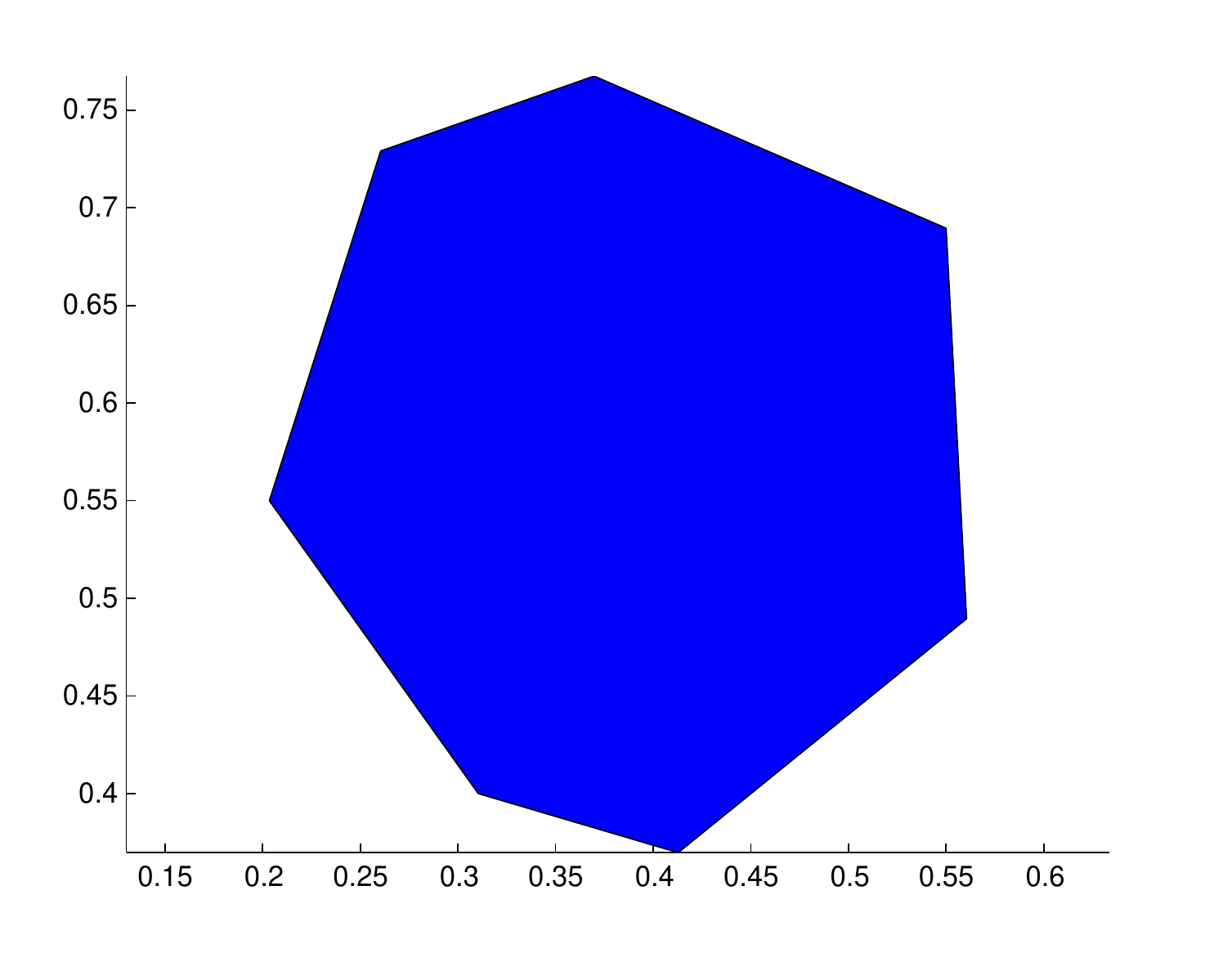}
\caption{}
\label{untitled63}
\end{figure}

\begin{figure}[htbp] \label{untitled26}
\centering
\includegraphics[width=.8 \textwidth]{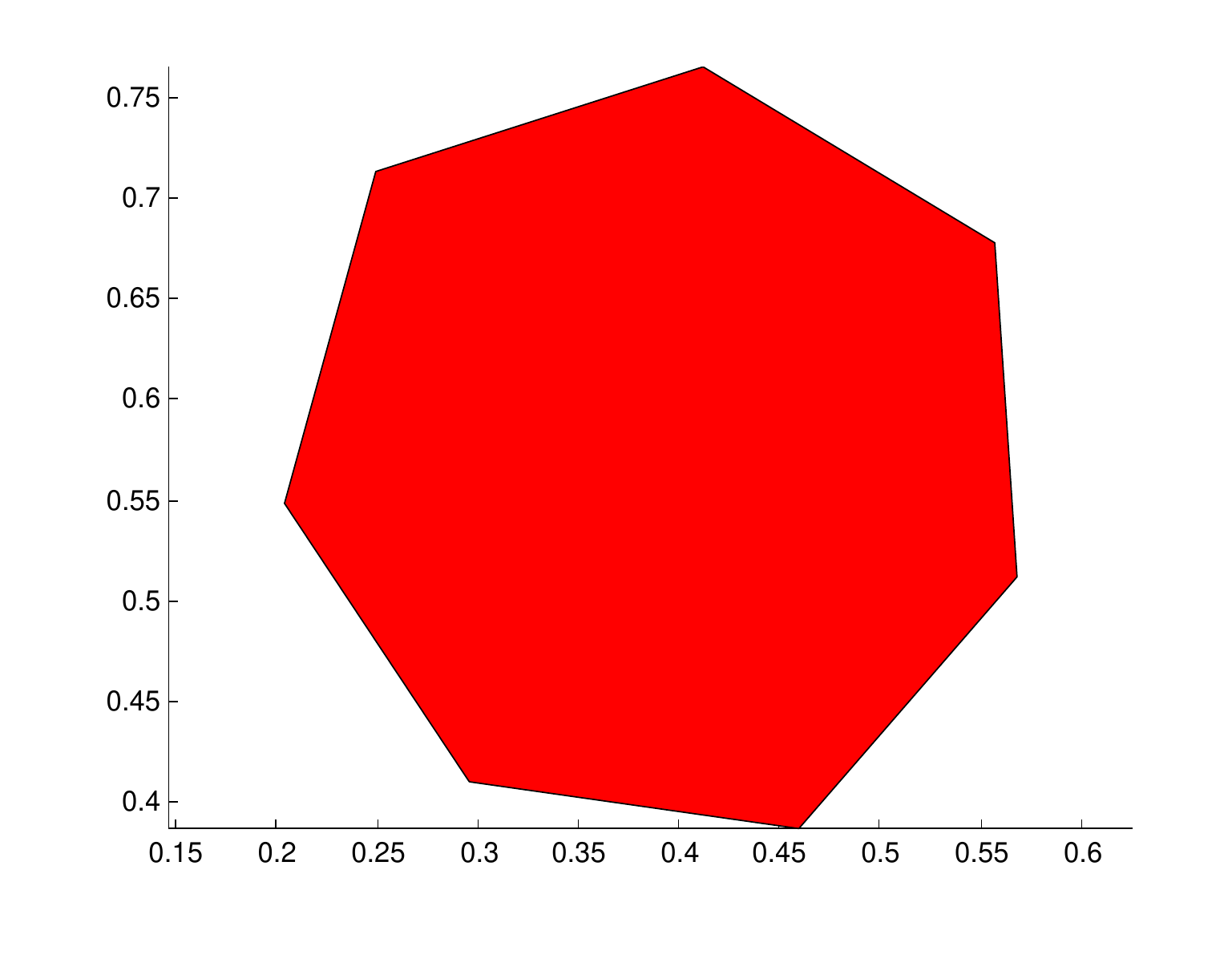}
\caption{}
\label{a12he}
\end{figure}

\begin{figure}[htbp] \label{untitled15}
\centering
\includegraphics[width=.8 \textwidth]{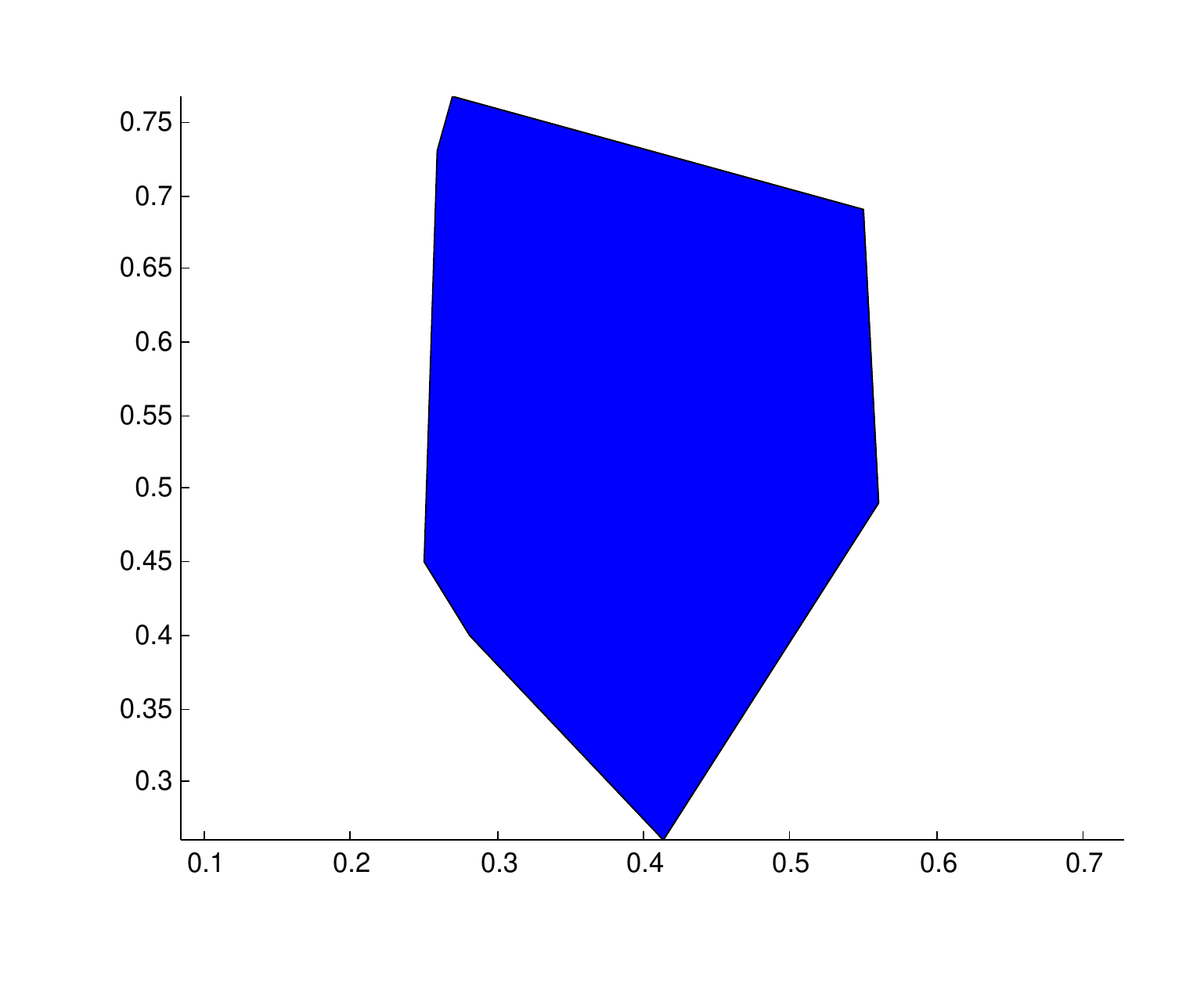}
\caption{}
\label{a12g}
\end{figure}

\begin{figure}[htbp] \label{untitled16}
\centering
\includegraphics[width=.8 \textwidth]{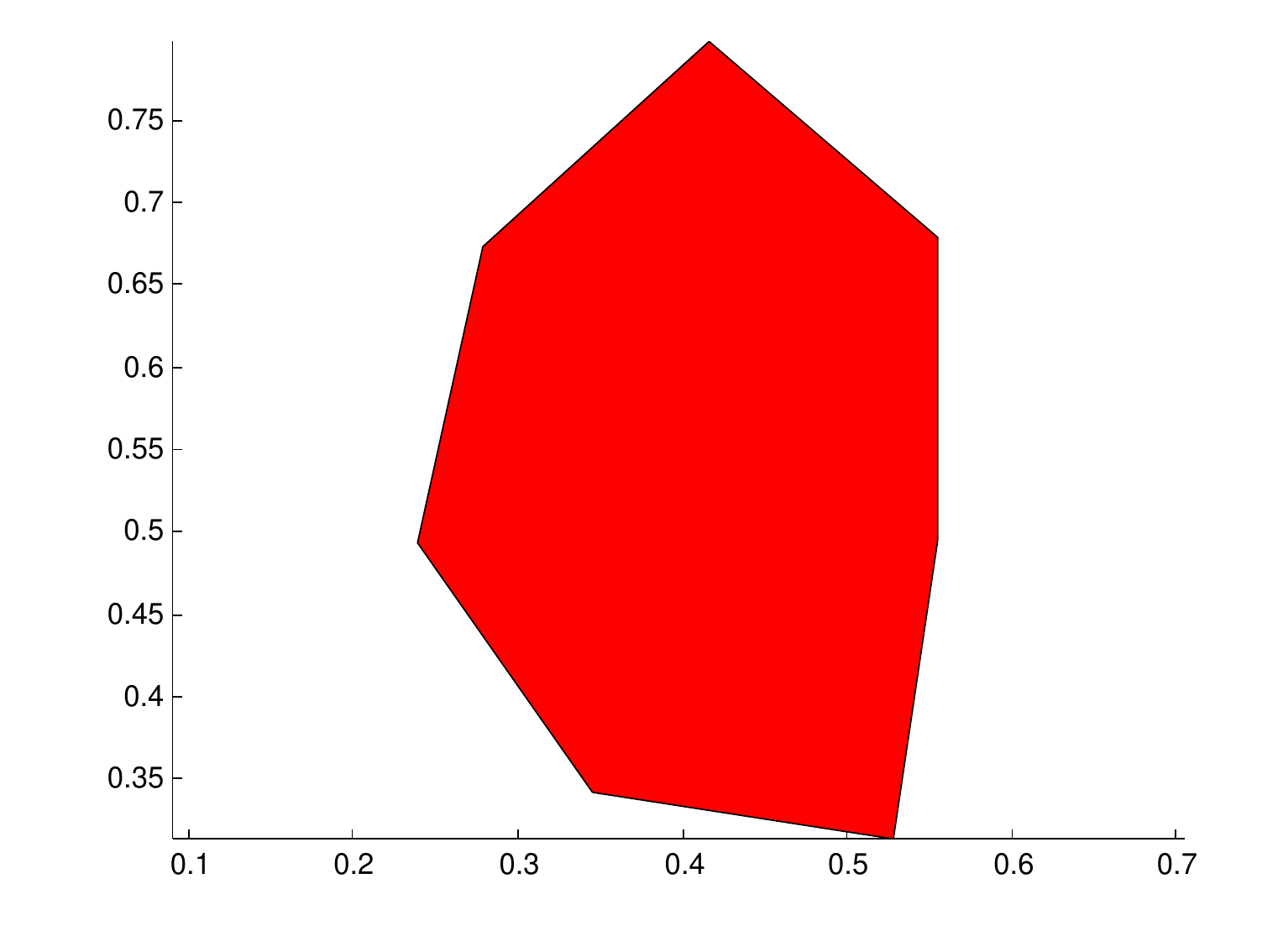}
\caption{}
\label{a12vf}
\end{figure}

The specific process of generating $P^\infty$ can be applied to any convex polygon. Therefore it is of interest to understand the specific flow. The following illuminates the algorithm if $n=7$: Figure \ref{untitled63} evolves into Figure \ref{a12he} in $27$ iterations (via MatLab). Note that the initial heptagon is not close to the regular convex heptagon, nevertheless the evolution is positive. Figure \ref{a12g} generates Figure  \ref{a12vf} after $55$ iterations (via MatLab) which is far from the regular cyclical heptagon, but it is almost equilateral. Observe that the initial heptagon in this case is far from the regular convex heptagon. Assuming $n$ is large, the cyclicity is imposed via $P_n \rightarrow B_1$, hence in the context of $n$ large, one has the expectation that many initial $n$-gons in a neighborhood of $P_n$ evolve to $P_n$. \\

\bibliographystyle{amsalpha}
\bibliography{References}

\end{document}